\documentclass[12pt,reqno]{amsart} 
\usepackage{amssymb}
\usepackage{amsthm}
\usepackage{amsmath}
\usepackage{verbatim}
\usepackage{enumerate}
\usepackage[usenames,dvipsnames] {color}
\numberwithin{equation}{section}
\newcommand\lam{\lambda}
\renewcommand\phi{\varphi}
\def\d{\mathbb{D}}
\def\c{\mathbb{C}}

\newcommand\ov{\overline}

\newcommand\inv{^{-1}}

\newcommand\bbm{\begin{bmatrix}}
\newcommand\ebm{\end{bmatrix}}
\newcommand\bpm{\begin{pmatrix}}
\newcommand\epm{\end{pmatrix}}

\def\m{\mathcal{M}}

\def\ph{\phi}

\newcommand\calm{\mathcal{M}}

\newcommand\df{\stackrel{\rm def}{=}}
\newcommand\ess{\mathcal{S}}

\def\be{\begin{equation}}
\def\ee{\end{equation}}

\def\ip#1#2{\langle #1,#2\rangle}

\def\s0{s_0}
\def\p0{p_0}

\DeclareMathOperator{\spn}{Span}

\theoremstyle{definition}
\newtheorem{defin}[equation]{Definition}
\newtheorem{lem}[equation]{Lemma}

\newtheorem{thm}[equation]{Theorem}

\newtheorem{example}[equation]{Example}
\newtheorem{remark}[equation]{Remark}

\begin{document}

\title[Models of holomorphic functions on the symmetrized skew bidisc]{Models of holomorphic functions on \newline the symmetrized skew bidisc}

\author{Connor Evans}
\address{School of Mathematics,  Statistics and Physics, Newcastle University, Newcastle upon Tyne
	NE\textup{1} \textup{7}RU, U.K.}
\email{CEvansMathematics@outlook.com}

\author{Zinaida A. Lykova}
\address{School of Mathematics,  Statistics and Physics, Newcastle University, Newcastle upon Tyne
	NE\textup{1} \textup{7}RU, U.K.}
\email{Zinaida.Lykova@ncl.ac.uk}

\author{N. J. Young}
\address{School of Mathematics, Statistics and Physics, Newcastle University, Newcastle upon Tyne NE1 7RU, U.K.}
\email{Nicholas.Young@ncl.ac.uk}

\date{30th March, 2026}


\thanks{Partially supported by the Engineering and Physical Sciences Research Council grants EP/N03242X/1 and  DTP21  EP/T517914/1. }



\begin{abstract}
The purpose of this paper is to develop the theory of holomorphic functions 
with modulus bounded by $1$ on  the symmetrized skew bidisc
\[ \mathbb{G}_{r} \stackrel{\rm def}{=}  \Big\{( \lambda_{1}+r\lambda_{2} ,r\lambda_{1}\lambda_{2}): \lambda_{1}\in \mathbb{D}, \lambda_{2}\in\mathbb{D}\Big\},
\]
for a fixed $r \in (0,1)$. We show the existence of a realization formula and a model formula for such holomorphic functions.

\end{abstract}

\maketitle

\section{Introduction}
In this paper we shall generalize some results from long-established function theory of the unit disc  $\mathbb{D}$  and from the theory of holomorphic functions on the bidisc $\mathbb{D}^{2}$  and  the symmetrized bidisc $\mathbb{G}$
 to holomorphic functions on  the symmetrized skew bidisc
$ \mathbb{G}_{r} $, for a fixed $r \in (0,1)$.

Recall that the {\em Schur class}, $\ess(\d)$, is the set of holomorphic functions $\ph$ on the unit disc $\d$ such that the supremum norm 
$\|\ph\|_\infty = \sup_{z\in\d}|\ph(z)|\leq 1$.
The notions of models and realizations of functions are useful for the understanding of the Schur class.
A {\em model} of a function $\ph:\d\to \c$ is a pair $(\m,u)$ where $\m$ is a Hilbert space and $u$ is a map from $\d$ to $\m$ such that, for all $\lam,\mu\in\d$,
\be\label{eqmodelD}
1-\ov{\ph(\mu)}\ph(\lam) = (1-\bar\mu \lam)\ip{u(\lam)}{u(\mu)}_\m,
\ee
where $\ip{\cdot}{\cdot}_\m$ denotes the inner product in $\m$.
A closely related notion is a {\em realization} of a function $\ph$ on $\d$, that is, a formula of the form
\be\label{realize-1}
\ph(\lam) =\alpha + \ip{\lam(1-D\lam)\inv \gamma}{\beta}_\m\qquad  \text{for all} \ \lam\in\d,
\ee
where $\bbm \alpha & 1\otimes \beta\\ \gamma\otimes 1 & D\ebm$ is the matrix of a  unitary operator on $\c\oplus \m$.

The connections between models, realizations and the Schur class are revealed in the following theorem.
\begin{thm}\label{1} Let $\ph$ be a function on $\d$.  The following conditions are equivalent.
\begin{enumerate}[(i)]
\item $\ph \in \ess(\d)$;
\item $\ph$ has a model;
\item $\ph$ has a realization.
\end{enumerate}
\end{thm}
Proofs of the various implications in this theorem can be found, for instance, in \cite{amy20}.
Models and realizations of functions have proved to be a powerful tool for both operator-theorists (e.g. Nagy and Foias \cite{NF}) and control engineers (largely as a tool for computation \cite{GL}).  In this paper we shall derive versions of model and realization formulae which apply  to functions in the ``Schur class" of another domain.  For a domain $\Omega$ in $\c^n$ the Schur class $\mathcal{S}(\Omega)$ is defined to be the set of holomorphic functions $\ph$ on $\Omega$ such that the supremum norm $\|\ph\|_\infty \df \sup_{z\in\Omega}|\ph(z)|$ is at most $1$.  We are concerned with the domain $\Omega =\mathbb{G}_r$ in $\c^2$,
which we now define.

The symmetrized bidisc $\mathbb{G}$ was introduced by Agler and Young  in the course of a study of the spectral Nevanlinna-Pick problem for $2 \times 2$ matrix functions, which is a special case of the ``$\mu$-synthesis problem" in robust control theory \cite{doyle}.  $\mathbb{G}$ is defined by
\be\label{defGr}
\mathbb{G}\stackrel{\rm def}{=}  \Big\{( \lambda_{1}+\lambda_{2} ,\lambda_{1}\lambda_{2}): \lambda_{1}\in \mathbb{D}, \lambda_{2}\in\mathbb{D}\Big\}.
\ee
It is known that $\mathbb{G}$ is hypoconvex, polynomially convex and starlike about $(0,0)$, but not convex, see \cite[Theorem 2.3]{HGSB}. 
Here we study a related region in $\c^2$, to wit,  the region
$$
 \mathbb{G}_{r} = \Big\{( \lambda_{1}+r\lambda_{2} ,r\lambda_{1}\lambda_{2}): \lambda_{1}\in \mathbb{D}, \lambda_{2}\in\mathbb{D}\Big\},
$$
where $0<r<1$.  Since $\mathbb{G}_{r}$ is the image of $\d\times r\d$ under the symmetrization map $(z,w) \mapsto (z+w,zw)$, and $\d\times r\d$ is also a bidisk, arguably $\mathbb{G}_{r}$ also deserves the appellation ``symmetrized bidisc".  However, this name has become firmly associated with the domain $\mathbb{G}$, and so we propose the nomenclature ``symmetrized skew bidisc" for $\mathbb{G}_{r}$, to avoid clashing with established terminology.  $\mathbb{G}_{r}$ is also potentially of interest in connection with the spectral Nevanlinna-Pick problem for $2\times 2$-matrix functions. 
 In a personal communication Lukasz Kosinski pointed out  that $\mathbb{G}_{r}$ is not pseudoconvex. 
 We shall also have occasion to make use of the domain 
\begin{align} \label{defrdotG}
 r \cdot \mathbb{G} &\df \Big\{(r(\lambda_{1}+\lambda_{2}),r^{2}\lambda_{1}\lambda_{2}): \lambda_{1}\in\mathbb{D}, \lambda_{2}\in\mathbb{D}\Big\}\\
 &=\Big\{(rs,r^2p):(s,p)\in\mathbb{G}\Big\}.
\end{align}
In 2017 Agler and Young \cite{AY17} derived a realization formula for any function in $\mathcal{S}(\mathbb{G})$ by means of a symmetrization argument. They introduced the following notion:
\begin{defin}\label{defGmodel}
A $\mathbb{G}$-\emph{model} for a function $\ph$ on $\mathbb{G}$ is a triple $(\m,T,u)$ where $\calm$ is a Hilbert space, $T$ is a contraction acting on $\calm$ and $u:\mathbb{G} \to \m$ is a holomorphic function such that, for all $s,t\in \mathbb{G}$,
\begin{equation}\label{modelform}
 1-\overline{\ph(t)}\ph(s)= \ip{ (1-t_T^* s_T) u(s)}{u(t)}_\calm.
\end{equation}
\end{defin}
Here, for any point $s=(s_1,s_2)\in \mathbb{G}$ and any contractive linear operator $T$ on a Hilbert space $\calm$, the operator $s_T$ is defined by
\begin{equation}\label{defsU}
s_T=(2s_2T-s_1)(2-s_1T)\inv \quad \mbox{ on } \calm.
\end{equation}
A {\em realization} of a function $\ph$ on $\mathbb{G}$ is a formula of the form
\be\label{realize-2}
\ph(s) =\alpha + \ip{s_T(1-D s_T)\inv \gamma}{\beta}_\m\qquad  \text{for all} \ s \in \mathbb{G},
\ee
where $\bbm \alpha & 1\otimes \beta\\ \gamma\otimes 1 & D\ebm$ is the matrix of a  unitary operator on $\c\oplus \m$ and $T$ is a contraction on $\calm$.

In \cite[Theorem 2.2 and Theorem 3.1]{AY17} Agler and Young proved the following statement.
\begin{thm}\label{modelGthm}
Let $\ph$ be a function on $\mathbb{G}$.  The following three statements are equivalent.
\begin{enumerate}
\item   $\ph\in\ess(\mathbb{G})$;
\item $\ph$ has a $\mathbb{G}$-model $(\calm, T, u)$ in which $T$ is a unitary operator on $\calm$;
\item $\ph$ has a realization.
\end{enumerate}
\end{thm}
To study $\mathbb{G}_{r}$, we define the involution $\sigma$ on $\mathbb{C}^{2}$ by 
\begin{equation}\label{invol-sigma-intro}
 \lambda^{\sigma} = (r\lambda_{2},r^{-1}\lambda_{1})~ \text{for all}~ \lambda= (\lambda_1, \lambda_2) \in \mathbb{C}^{2}.
\end{equation}
We perform  a symmetrization argument on $\d^2$ using the involution $\sigma$
to obtain a model formula for $\mathbb{G}_{r}$ in Theorem \ref{modelformulaforGr} and Theorem \ref{modelforsur}. To state the formulae we require the following notation.

\begin{defin}\label{sUR-def-intro}
Let $r\in(0,1)$, let $\mathcal{M}$ be a complex Hilbert space, let $\mathcal{H}_{1}$ be a closed non-trivial proper subspace of $\mathcal{M}$,  and let $U$ be a unitary operator on $\mathcal{M}$.  We define $\mathcal{R}$ in $\mathcal{B(M)}$ by the formula 
\begin{equation} \label{RonM-intro}
\mathcal{R}=\begin{bmatrix}
	1_{\mathcal{H}_{1}  } 	& 0                \\
	0 		& r\cdot 1_{\mathcal{H}_{1}^{\perp}}  	\\
\end{bmatrix}\in\mathcal{B(M)}.
\end{equation}
For $s=(s_{1},s_{2})\in r\cdot\mathbb{G}$, we define $s_{U,\mathcal{R}}\in\mathcal{B(M)}$  by
\begin{equation}\label{sUR-form-intro}
s_{U,\mathcal{R}} = \bigg(2s_{2}\mathcal{R}^{-1}U-s_{1}\bigg)\bigg(2\mathcal{R}-s_{1}U\bigg)^{-1}.
\end{equation}
\end{defin}

\begin{remark} Let $r\in(0,1)$.
The relation between the operator $s_{U,\mathcal{R}}\in\mathcal{B(M)}$ given by equation \eqref{sUR-form-intro}
and  the operator $s_T\in\mathcal{B(M)}$ given by equation \eqref{defsU} is the following.
For $s=(s_{1},s_{2})\in r\cdot\mathbb{G}$,
\begin{equation} \label{sUR-sT}
s_{U,\mathcal{R}} = s_{\mathcal{R}^{-1}U}\mathcal{R}^{-1}.
\end{equation}
Note that $\|\mathcal{R}^{-1}U\| = r^{-1}$, and so $\mathcal{R}^{-1}U$ is not a contraction, but 
one can check that, for $s=(s_{1},s_{2})\in r\cdot\mathbb{G}$, the operator $s_{\mathcal{R}^{-1}U}$ is still well defined.
\end{remark}

We prove the following results in Lemma \ref{sURanalyticlemma}.
\begin{lem}\label{sURanalyticlemma-intro} Let $r\in(0,1)$, let $\mathcal{M}$ be a complex Hilbert space, let $\mathcal{H}_{1}$ be a closed non-trivial proper subspace of $\mathcal{M}$, let 
the operator $\mathcal{R} \in\mathcal{B(M)}$ be defined by equation \eqref{RonM-intro} and $U$ be a unitary operator on $\mathcal{M}$.
\begin{enumerate}
    \item The operator-valued function 
     $$w:r\cdot\mathbb{G}\rightarrow \mathcal{B(M)} : s\mapsto s_{U,\mathcal{R}},$$ 
where $s_{U,\mathcal{R}}\in\mathcal{B(M)}$ is given by equation \eqref{sUR-form-intro},   
is well defined and holomorphic on $r\cdot\mathbb{G}$;
    \item $\|s_{U,\mathcal{R}}\|_{\mathcal{B(M)}}< 1$ for all $s=(s_{1},s_{2})\in r\cdot\mathbb{G}$.
\end{enumerate}
\end{lem}

\begin{thm}\label{modelforsur-intro}
Let $r\in(0,1)$ and let $f\in\mathcal{S}(\mathbb{G}_{r})$. Then there exists a model $(\mathcal{M}, (U,\mathcal{R}), u)$ for $f$ on $r\cdot \mathbb{G}$, that is, there exist a complex Hilbert space $\mathcal{M}$,  a closed non-trivial proper subspace  $\mathcal{H}_{1}$ of $\mathcal{M}$, a holomorphic map $u:r\cdot \mathbb{G}\rightarrow \mathcal{M}$, a unitary operator $U$ on $\mathcal{M}$ and 
the operator $\mathcal{R} \in\mathcal{B(M)}$ given by equation \eqref{RonM-intro},
such that, for all $s=(s_{1},s_{2})\in r\cdot\mathbb{G}$ and $t=(t_{1},t_{2})\in r \cdot \mathbb{G}$,
\begin{equation}\label{grcorollary-intro}
1- \overline{f(t)}f(s) =\Bigg\langle  \bigg(1_{\mathcal{M}}-t_{U,\mathcal{R}}^{*}s_{U,\mathcal{R}}\bigg)u(s),u(t)\Bigg\rangle_{\mathcal{M}},
  \end{equation}
where the operators $s_{U,\mathcal{R}}$ and $t_{U,\mathcal{R}}$ are defined by 
equation \eqref{sUR-form-intro}.
\end{thm}

Note that the model formula of a function $f\in\mathcal{S}(\mathbb{G}_{r})$ is similar to the model formula 
\eqref{defGmodel} of a function 
$f\in\mathcal{S}(\mathbb{G})$ except that the operators $s_U, t_U$  are replaced by the operators $s_{U,\mathcal{R}}$ and $t_{U,\mathcal{R}}$ respectively, where $R \in \mathcal{B(M)}$ given by equation \eqref{RonM-intro}. 

We  prove in Theorem \ref{realisationformulaforrG} a realization formula for functions in $\mathcal{S}(\mathbb{G}_{r})$.
Let us state this result. 
\begin{thm}\label{theoremin-intro}
Let $ r\in (0,1)$ and $f \in \mathcal{S}(\mathbb{G}_{r})$. There exist a scalar $a\in\mathbb{C}$,  a complex Hilbert space $\mathcal{M}$, vectors $\beta, \gamma, \in \mathcal{M}$, a closed non-trivial proper subspace $\mathcal{H}_{1}$ of $\mathcal{M}$ and linear operators $D,U$ on $\mathcal{M}$ such that $D$ is a contraction, $U$ is unitary such that the operator 
\begin{equation}
L=\begin{bmatrix}
    a       & 1\otimes \beta \\
    \gamma \otimes 1       & D 
\end{bmatrix}
\end{equation}
is unitary on $\mathbb{C}\oplus \mathcal{M}$ and, for all $s=(s_{1},s_{2})\in r\cdot\mathbb{G}$,
$$f(s)=a+\langle s_{U,\mathcal{R}}(1-Ds_{U,\mathcal{R}})^{-1}\gamma, \beta \rangle_{\mathcal{M}},$$
where  the operator $s_{U,\mathcal{R}}$ is defined by 
equation \eqref{sUR-form-intro} and the operator $\mathcal{R} \in\mathcal{B(M)}$ given by equation \eqref{RonM-intro}.
\end{thm}

\section{A model formula for the bidisc $\mathbb{D}^{2}$ and relations to  the symmetrized skew bidisc}\label{symmetrizedpolydiscandGr}

As a preliminary to the construction of models of functions on $\mathbb{G}_{r}$, we recall the notion of a Hilbert space model of a function on $\mathbb{D}^{2}$.

\begin{defin}{\normalfont{\cite[Definition 4.18]{amy20}}}\label{modelonbidisc}
 Let $\varphi$ be a function on $\mathbb{D}^{2}$. A pair $(\mathcal{H},u)$ is said to be a model of $\varphi$ if $\mathcal{H}=\mathcal{H}_{1}\oplus\mathcal{H}_{2}$ is a Hilbert space, $\mathcal{H}_{1}$ and $\mathcal{H}_{2}$ are orthogonally complementary subspaces of $\mathcal{H}$ and $u=(u_{1},u_{2})$ is a pair of holomorphic maps from $\mathbb{D}^{2}$ to $\mathcal{H}_{1},\mathcal{H}_{2}$ respectively such that, for all $\lambda=(\lambda_1, \lambda_2), \ \mu= (\mu_1, \mu_2) \in \mathbb{D}^{2}$, 
\begin{equation}\label{modelonbidiscformula}
1-\overline{\varphi(\mu)}\varphi(\lambda) = \big\langle(1-\overline{\mu_{1}}\lambda_{1}) u_{1}({\lambda}),u_{1}({\mu})\big\rangle_{\mathcal{H}_{1}}+\big\langle(1-\overline{\mu_{2}}\lambda_{2}) u_{2}({\lambda}),u_{2}(\mu)\big\rangle_{\mathcal{H}_{2}}.
\end{equation}
\end{defin}

It was proved by Agler in \cite{RepresentationOfCertainHolomorphicFunctions} that any holomorphic function $\varphi: \mathbb{D}^{2} \to \overline{\mathbb{D}}$ has a model.

\begin{thm}{\normalfont{({\bf{Agler}})}}\label{aglerstheorembidisc}
A function $\phi$ on $\mathbb{D}^{2}$ belongs to the Schur class  $\mathcal{S}(\mathbb{D}^{2})$ if and only if $\phi$ has a model. 
\end{thm}

To study $\mathbb{G}_{r}$, we define the involution $\sigma$ on $\mathbb{C}^{2}$ by 
\begin{equation}\label{invol-sigma}
 \lambda^{\sigma} = (r\lambda_{2},r^{-1}\lambda_{1})~ \text{for all}~ \lambda= (\lambda_1, \lambda_2) \in \mathbb{C}^{2}.
 \end{equation}
Note that,  for all $\lambda\in r\mathbb{D}\times \mathbb{D}$, we have $ \lambda^{\sigma}\in r\mathbb{D}\times \mathbb{D}$ and
\begin{equation}\label{sigmainvolutionequation1}
(\lambda^{\sigma})^{\sigma} = (r\lambda_{2},r^{-1}\lambda_{1})^{\sigma} = (rr^{-1}\lambda_{1},r^{-1}r\lambda_{2}) = \lambda.
\end{equation}
This implies $(r\mathbb{D}\times\mathbb{D})^{\sigma} = r\mathbb{D}\times \mathbb{D}.$
Define the operator  $T_{r}:\mathbb{C}^{2} \to \mathbb{C}^{2}$ by 
\begin{equation}\label{Tr-def}
T_{r}(\lambda_{1},\lambda_{2})= (\lambda_{1},r\lambda_{2}) \ \text{for} \ \lambda=(\lambda_{1},\lambda_{2})\in\mathbb{C}^{2}.
\end{equation}
Define also the map $\pi : \mathbb{C}^{2} \rightarrow \mathbb{C}^{2}$ by the formula
\begin{equation}\label{pi-def}
\pi(\lambda_{1},\lambda_{2}) = (\lambda_{1}+\lambda_{2}, \lambda_{1}\lambda_{2}) \ \text{for} \ (\lambda_1, \lambda_2) \in \mathbb{C}^{2},
\end{equation}
so that we have $\mathbb{G}_{r} = \pi(\mathbb{D}\times r\mathbb{D})$. 
Note that, for $\lambda = (r\lambda_{1},\lambda_{2}) \in r\mathbb{D}\times \mathbb{D}$, $\lambda^{\sigma}=(r\lambda_{2},\lambda_{1})$ and
\begin{align*}
& \pi\big(T_{r}(\lambda)\big) = \pi(r\lambda_{1},r\lambda_{2}) = (r(\lambda_{1}+\lambda_{2}),r^{2}\lambda_{1}\lambda_{2}),\\
& \pi\big(T_{r}(\lambda^{\sigma})\big) = \pi\big(T_{r}(r\lambda_{2},\lambda_{1})\big) = \pi(r\lambda_{2},r\lambda_{1})= (r(\lambda_{1}+\lambda_{2}),r^{2}\lambda_{1}\lambda_{2}).\\ 
\end{align*}
Thus, for all $\lambda\in r\mathbb{D}\times\mathbb{D}$,
\begin{equation}\label{piTrequation}
\pi\big(T_{r}(\lambda)\big) =  \pi\big(T_{r}(\lambda^{\sigma})\big).
\end{equation}
Let $f : \mathbb{G}_{r} \to \overline{\mathbb{D}}$ be a holomorphic function. Then we may define $F:\mathbb{D}^{2}\rightarrow\overline{\mathbb{D}}$ by
\begin{equation}\label{F=fpiTr}
 F = f \circ \pi \circ T_{r} : \mathbb{D}^{2} \to \overline{\mathbb{D}}.
 \end{equation}
It  is clear  that $F$ is in the Schur class of $\mathbb{D}^{2}$. Note that, by equation (\ref{piTrequation}), $F$ is symmetric with respect to the involution $\sigma$,
\begin{equation}\label{Finvolution}
 F(\lambda^{\sigma}) = f(\pi(r\lambda_{2},\lambda_{1})) = f(\lambda_{1}+r\lambda_{2},r\lambda_{1}\lambda_{2}) = F(\lambda), \ \text{for all} \
  \lambda \in  r\mathbb{D}\times\mathbb{D}.
\end{equation}
We now bring all these notions together with the model of a function on $\mathbb{D}^{2}$ to prove the following statement.
\begin{thm}\label{modelformulaforGr}
Let $f\in \mathrm{Hol}(\mathbb{G}_{r},\overline{\mathbb{D}})$ and let 
$$F = f \circ \pi \circ T_{r}: \mathbb{D}^{2} \to \overline{\mathbb{D}}.$$ Then 
there exist a complex Hilbert space $\mathcal{M}$,  a closed non-trivial proper subspace  $\mathcal{H}_{1}$ of $\mathcal{M}$,
a unitary operator $U$ on $\mathcal{M}$, 
 a holomorphic map $w: r\mathbb{D}\times \mathbb{D} \to \mathcal{M}$, which satisfies $w(\lambda^{\sigma}) = w(\lambda)$ for all $\lambda \in r\mathbb{D}\times\mathbb{D}$,   such that, for all $\lambda, \mu \in r\mathbb{D}\times\mathbb{D}$, 
\begin{equation}\label{Zrinnerproduct}
1-\overline{F(\mu)}F(\lambda) = \langle Z_{r}(\lambda,\mu)w(\lambda),w(\mu) \rangle_{\mathcal{M}}, 
\end{equation}
where
\begin{align*}
Z_{r}(\lambda,\mu) =& (1_{\mathcal{M}}-r\overline{\mu_{2}}\mathcal{R}^{-1}U^{*})(1_{\mathcal{M}}-  \overline{\mu}_{1}\lambda_{1}\mathcal{R}^{-2})(1_{\mathcal{M}}-r\lambda_{2}U\mathcal{R}^{-1}) \nonumber \\
+ & (1_{\mathcal{M}}-\overline{\mu_{1}}\mathcal{R}^{-1}U^{*})(1_{\mathcal{M}}-r^{2} \overline{\mu}_{2}\lambda_{2}\mathcal{R}^{-2})(1_{\mathcal{M}}-\lambda_{1}U\mathcal{R}^{-1})
\end{align*}
and $\mathcal{R}\in\mathcal{B(M)}$ is defined by equation \eqref{RonM-intro}.
\end{thm}
\begin{proof}
Since $F \in \mathcal{S}(\mathbb{D}^{2})$, by Agler's Theorem \ref{aglerstheorembidisc}, $F$ has a model $(\mathcal{H},u)$, that is, there exists an orthogonally decomposed Hilbert space $\mathcal{H}=\mathcal{H}_{1}\oplus\mathcal{H}_{2}$ and a pair of holomorphic maps $u=(u_{1},u_{2})$ from $\mathbb{D}^{2}$ to $\mathcal{H}_{1},\mathcal{H}_{2}$ respectively such that, for all $\lambda, \mu \in \mathbb{D}^{2}$, 
\begin{equation}\label{bidiscmodel}
1-\overline{F(\mu)}{F(\lambda)}=\langle(1-\overline{\mu_{1}}\lambda_{1}) u_{1}({\lambda}),u_{1}({\mu})\rangle_{\mathcal{H}_{1}}+\langle(1-\overline{\mu}_{2}\lambda_{2}) u_{2}({\lambda}),u_{2}({\mu})\rangle_{\mathcal{H}_{2}}.
\end{equation}

Consider $\lambda$ and $\mu$ in $ r\mathbb{D}\times \mathbb{D}$,
replace $\lambda, \mu$ by $\lambda^{\sigma}, \mu^{\sigma}$ respectively in equation \eqref{bidiscmodel} and use equation \eqref{Finvolution} to deduce that, for all $\lambda$ and $ \mu $ in $ r\mathbb{D}\times \mathbb{D}$, the following equation holds
\begin{align}\label{bidiscmodelsigma}
1-&\overline{F(\mu)}{F(\lambda)} \nonumber\\
&=(1-r^{2}\overline{\mu_{2}}\lambda_{2})\langle{u_{1}({\lambda^{\sigma}}),u_{1}({\mu}^{\sigma})\rangle_{\mathcal{H}_{1}}}+\langle{(1-r^{-2}\overline{\mu_{1}}\lambda_{1})u_{2}({\lambda^{\sigma}}),u_{2}({\mu^{\sigma}})\rangle_{\mathcal{H}_{2}}}.
\end{align}
Take the average of equations (\ref{bidiscmodel}) and (\ref{bidiscmodelsigma}) to obtain, for  all $\lambda$ and $\mu$ in $r\mathbb{D}\times\mathbb{D}$, 
\begin{align*}
1-& \overline{F(\mu)}F(\lambda) \\
= & \dfrac{1}{2}\Bigg(\Big\langle(1-\overline\mu_{1}\lambda_{1})u_{1}({\lambda}),u_{1}(\mu)\Big\rangle_{\mathcal{H}_{1}}+ \Big\langle (1-r^{-2}\overline\mu_{1}\lambda_{1})u_{2}({\lambda^{\sigma}}), u_{2}(\mu^{\sigma})\Big\rangle_{\mathcal{H}_{2}} \\
&+  \Big\langle(1-r^{2}\overline{\mu}_{2}\lambda_{2}) u_{1}(\lambda^{\sigma}), u_{1}(\mu^{\sigma})\Big\rangle_{\mathcal{H}_{1}}+\Big\langle(1-\overline{\mu}_{2}\lambda_{2})u_{2}(\lambda), u_{2}(\mu)\Big\rangle_{\mathcal{H}_{2}}\Bigg). 
\end{align*}
The last equation can be  be re-written as
\begin{align}
1-\overline{F(\mu)}F(\lambda) =& \dfrac{1}{2}\Bigg(\Bigg\langle
\begin{bmatrix}
	(1-\overline\mu_{1}\lambda_{1})u_{1}(\lambda) \\
	(1-r^{-2}\overline\mu_{1}\lambda_{1})u_{2}(\lambda^{\sigma}) \nonumber \\
\end{bmatrix},
\begin{bmatrix}
           u_{1}(\mu) \\
           u_{2}(\mu^{\sigma}) \\
\end{bmatrix}\Bigg\rangle_{\mathcal{H}_{1}\oplus \mathcal{H}_{2}}\\
& + \Bigg\langle
\begin{bmatrix}
	(1-r^{2}\overline\mu_{2}\lambda_{2})u_{1}({\lambda^{\sigma}}) \\
	(1-\overline\mu_{2}\lambda_{2})u_{2}({\lambda}) \\
\end{bmatrix},
\begin{bmatrix}
	u_{1}({\mu}^{\sigma}) \\
	u_{2}({\mu}) \\
\end{bmatrix}\Bigg\rangle_{\mathcal{H}_{1}\oplus \mathcal{H}_{2}}\Bigg). \label{bulkyequation}
\end{align}
For each $\lambda \in r\mathbb{D}\times\mathbb{D}$, define
the vector $v(\lambda)\in\mathcal{H}$ and the operator $\widetilde{\mathcal{R}} \in \mathcal{B(H)}$ by
\begin{equation}\nonumber
v({\lambda}) = \dfrac{1}{\sqrt{2}}
\begin{bmatrix}
	u_{1}({\lambda}) \\
	u_{2}({\lambda^{\sigma}}) \\
\end{bmatrix},~
\widetilde{\mathcal{R}} = 
\begin{bmatrix}
	1_{\mathcal{H}_1}   	& 0                \\
	0 		& r\cdot 1_{\mathcal{H}_2}  	\\
\end{bmatrix}.
\end{equation}
Then, for all $\lambda,\mu \in r\mathbb{D}\times\mathbb{D}$, equation (\ref{bulkyequation}) can be written as
\begin{equation}\label{averagedequations}
1-\overline{F(\mu)}F(\lambda) = \Big\langle(1_{\mathcal{H}}-\overline{\mu}_{1}\lambda_{1}\widetilde{\mathcal{R}}^{-2})v({\lambda}),v({\mu})\Big\rangle_{\mathcal{H}}+\Big\langle(1_{\mathcal{H}}-r^{2}\overline{\mu}_{2}\lambda_{2}\widetilde{\mathcal{R}}^{-2})v({\lambda}^{\sigma}),v(\mu^{\sigma})\Big\rangle_{\mathcal{H}}.
\end{equation}
Again, use the fact that $F(\lambda^{\sigma})=F(\lambda)$ for all $\lambda\in r \mathbb{D}\times\mathbb{D}$ and replace $\lambda$ with $\lambda^{\sigma}$ in equation (\ref{averagedequations}) to obtain
\begin{equation}\label{newaveragedequations}
1-\overline{F(\mu)}F(\lambda) = \Big\langle(1_{\mathcal{H}}-r\overline{\mu}_{1}\lambda_{2}\widetilde{\mathcal{R}}^{-2})v(\lambda^{\sigma}),v(\mu)\Big\rangle_{\mathcal{H}}+\Big\langle(1_{\mathcal{H}}-r\overline{\mu}_{2}\lambda_{1}\widetilde{\mathcal{R}}^{-2})v({\lambda}),v({\mu}^{\sigma})\Big\rangle_{\mathcal{H}}.
\end{equation}
We then equate the right hand sides of equations (\ref{averagedequations}) and (\ref{newaveragedequations}) to see that
\begin{align*}
\Big\langle(1_{\mathcal{H}}-\overline{\mu}_{1}&\lambda_{1}\widetilde{\mathcal{R}}^{-2})v({\lambda}),v({\mu})\Big\rangle_{\mathcal{H}}+\Big\langle(1_{\mathcal{H}}-r^{2}\overline{\mu}_{2}\lambda_{2}\widetilde{\mathcal{R}}^{-2})v({\lambda}^{\sigma}),v({\mu}^{\sigma})\Big\rangle_{\mathcal{H}}\nonumber \\ 
= \Big\langle(1_{\mathcal{H}}-&r\overline{\mu}_{1}\lambda_{2}\widetilde{\mathcal{R}}^{-2})v(\lambda^{\sigma}),v(\mu)\Big\rangle_{\mathcal{H}}+\Big\langle(1_{\mathcal{H}}-r\overline{\mu}_{2}\lambda_{1}\widetilde{\mathcal{R}}^{-2})v({\lambda}),v({\mu}^{\sigma})\Big\rangle_{\mathcal{H}}. 
\end{align*}
Expanding brackets, we find that
\begin{align}
& \Big\langle v(\lambda), v(\mu) \Big\rangle_{\mathcal{H}} -  \Big\langle \overline{\mu}_{1}\lambda_{1} \widetilde{\mathcal{R}}^{-2} v(\lambda), v(\mu)\Big\rangle_{\mathcal{H}} \nonumber\\
& \hspace{1cm} + \Big\langle v(\lambda^{\sigma}), v(\mu^{\sigma})\Big\rangle_{\mathcal{H}} - \Big\langle r^{2}\overline{\mu}_{2}\lambda_{2}\widetilde{\mathcal{R}}^{-2}v(\lambda^{\sigma}), v(\mu^{\sigma}) \Big\rangle_{\mathcal{H}} \nonumber \\
& =\Big\langle v(\lambda^{\sigma}), v(\mu) \Big\rangle_{\mathcal{H}} - \Big\langle r\overline{\mu}_{1}\lambda_{2}\widetilde{\mathcal{R}}^{-2}v(\lambda^{\sigma}), v(\mu) \Big\rangle_{\mathcal{H}}
\nonumber\\
& \hspace{1cm} + \Big\langle v(\lambda), v(\mu^{\sigma})\Big\rangle_{\mathcal{H}}-\Big\langle r\overline{\mu}_{2}\lambda_{1}\widetilde{\mathcal{R}}^{-2}v(\lambda),v(\mu^{\sigma})\Big\rangle_{\mathcal{H}} \label{rearrangingthisequation}.
\end{align}
Rearrange equation (\ref{rearrangingthisequation}) to obtain, for all $\lambda,\mu\in r\mathbb{D}\times \mathbb{D}$,
\begin{align*}
\Big\langle v(\lambda), v(\mu) \Big\rangle_{\mathcal{H}} &+ \Big\langle v(\lambda^{\sigma}),v(\mu^{\sigma})\Big\rangle_{\mathcal{H}}
- \Big\langle v(\lambda^{\sigma}), v(\mu)\Big\rangle_{\mathcal{H}} - \Big\langle v(\lambda), v(\mu^{\sigma})\Big\rangle_{\mathcal{H}} \\
=&\Big\langle \overline{\mu}_{1}\lambda_{1} \widetilde{\mathcal{R}}^{-2} v(\lambda), v(\mu)\Big\rangle_{\mathcal{H}} +\Big\langle r^{2}\overline{\mu}_{2}\lambda_{2}\widetilde{\mathcal{R}}^{-2}v(\lambda^{\sigma}), v(\mu^{\sigma}) \Big\rangle_{\mathcal{H}} \\
&- \Big\langle r\overline{\mu}_{1}\lambda_{2}\widetilde{\mathcal{R}}^{-2}v(\lambda^{\sigma}), v(\mu) \Big\rangle_{\mathcal{H}} -\Big\langle r\overline{\mu}_{2}\lambda_{1}\widetilde{\mathcal{R}}^{-2}v(\lambda),v(\mu^{\sigma})\Big\rangle_{\mathcal{H}}.
\end{align*}
The last  equation can be simplified to 
\begin{align*}
\Big\langle v(\lambda) - v(\lambda^{\sigma}), & \ v(\mu)\Big\rangle_{\mathcal{H}} + \Big\langle v(\lambda^{\sigma})-v(\lambda), v(\mu^{\sigma})\Big\rangle_{\mathcal{H}} \\
 =&\Big\langle \overline{\mu}_{1}\lambda_{1} \widetilde{\mathcal{R}}^{-2} v(\lambda)-r\overline{\mu}_{1}\lambda_{2}\widetilde{\mathcal{R}}^{-2}v(\lambda^{\sigma}), v(\mu)\Big\rangle_{\mathcal{H}} \\
&+ \Big\langle r^{2}\overline{\mu}_{2}\lambda_{2}\widetilde{\mathcal{R}}^{-2}v(\lambda^{\sigma}) - r\overline{\mu}_{2}\lambda_{1}\widetilde{\mathcal{R}}^{-2}v(\lambda),v(\mu^{\sigma}) \Big\rangle_{\mathcal{H}} 
\end{align*}
and then to
\begin{align}\label{right-form-1}
\Big\langle v(\lambda) - v(\lambda^{\sigma}), & \ v(\mu)-v(\mu^{\sigma})\Big\rangle_{\mathcal{H}} \nonumber \\
=&\Big\langle \overline{\mu}_{1}\widetilde{\mathcal{R}}^{-2}\big(\lambda_{1} v(\lambda)-r\lambda_{2}v(\lambda^{\sigma})\big), v(\mu)\Big\rangle_{\mathcal{H}} \nonumber \\
&+\Big\langle r\overline{\mu}_{2}\widetilde{\mathcal{R}}^{-2}\big(r\lambda_{2}v(\lambda^{\sigma}) - \lambda_{1}v(\lambda)\big),v(\mu^{\sigma}) \Big\rangle_{\mathcal{H}} .
\end{align}
The equation \eqref{right-form-1} can then be written in the form
\begin{align*}
\Big\langle v(\lambda) - v(\lambda^{\sigma}),& \ v(\mu)-v(\mu^{\sigma})\Big\rangle_{\mathcal{H}} \\
=&\Big\langle \widetilde{\mathcal{R}}^{-1}\big(\lambda_{1} v(\lambda)-r\lambda_{2}v(\lambda^{\sigma})\big), \mu_{1}\widetilde{\mathcal{R}}^{-1}v(\mu)\Big\rangle_{\mathcal{H}} \\
&+\Big\langle \widetilde{\mathcal{R}}^{-1}\big(r\lambda_{2}v(\lambda^{\sigma}) - \lambda_{1}v(\lambda)\big),r\mu_{2}\widetilde{\mathcal{R}}^{-1}v(\mu^{\sigma}) \Big\rangle_{\mathcal{H}} 
\end{align*}
and further simplified to the form
\begin{align*}
\Big\langle v(\lambda) - v(\lambda^{\sigma}), & \ v(\mu)-v(\mu^{\sigma})\Big\rangle_{\mathcal{H}} \\
=&\Big\langle \widetilde{\mathcal{R}}^{-1}\big(\lambda_{1} v(\lambda)-r\lambda_{2}v(\lambda^{\sigma})\big), \widetilde{\mathcal{R}}^{-1}\big(\mu_{1}v(\mu)-r\mu_{2}v(\mu^{\sigma})\big)\Big\rangle_{\mathcal{H}}.
\end{align*}
This is equivalent to saying that the Gramian of the family 
$$\{v({\lambda})-v({\lambda}^{\sigma}):\lambda\in r\mathbb{D}\times\mathbb{D}\}$$ in $\mathcal{H}$ is equal to the Gramian of the family 
$$\{\widetilde{\mathcal{R}}^{-1}(\lambda_{1}v({\lambda})-r\lambda_{2}v({\lambda}^{\sigma})) : \lambda \in r\mathbb{D}\times\mathbb{D}\},$$ also in $\mathcal{H}$. Hence
there exists a linear isometry 
\begin{align*}
 L :  \overline{\spn}\Big\{\widetilde{\mathcal{R}}^{-1}(\lambda_{1}v({\lambda})- r\lambda_{2}v({\lambda}^{\sigma})) & : \lambda \in r\mathbb{D}\times\mathbb{D}\Big\} \\
& \rightarrow \overline{\spn}\Big\{v({\lambda})-v({\lambda}^{\sigma}) : \lambda\in r\mathbb{D}\times\mathbb{D}\Big\}
\end{align*}
with
\begin{equation}\label{isometryonthegramian}
L\Big( \widetilde{\mathcal{R}}^{-1}(\lambda_{1}v({\lambda})-r\lambda_{2}v({\lambda}^{\sigma})\Big) = v({\lambda})-v({\lambda}^{\sigma}),
\end{equation}
for all $\lambda \in r\mathbb{D}\times \mathbb{D}$. For subsequent calculations, it becomes advantageous to extend $L$ to a unitary operator $U$ on a Hilbert space $\mathcal{M}\supseteq \mathcal{H}$. We also extend $\widetilde{\mathcal{R}}$ to an operator $\mathcal{R}$ on the Hilbert space $\mathcal{M}=\mathcal{H}_{1}\oplus \mathcal{H}_{1}^{\perp}$, where $\mathcal{H}_{1}^{\perp}= \mathcal{M} \ominus \mathcal{H}_{1}$,  by
\begin{equation*}
\mathcal{R} = 
\begin{bmatrix}
	\widetilde{\mathcal{R}}_{\mathcal{H}}   	& 0                \\
	0 		& r_{\mathcal{H}^{\perp}}  	\\
\end{bmatrix} = 
\begin{bmatrix}
	1_{\mathcal{H}_{1}}   	& 0			& 0              \\
	0 		& r\cdot 1_{\mathcal{H}_{2}}		& 0		\\
	0		& 0					& r\cdot 1_{\mathcal{H}^{\perp}}	\\
\end{bmatrix} = 
\begin{bmatrix}
	1_{\mathcal{H}_{1}  } 	& 0                \\
	0 		& r\cdot 1_{\mathcal{H}_{1}^{\perp}}  	\\
\end{bmatrix}.
\end{equation*}
We rearrange equation (\ref{isometryonthegramian}) with $L$ replaced by $U$ to obtain 
\begin{equation}\label{equalitybetweenextension}
(1_{\mathcal{M}}-\lambda_{1}U\mathcal{R}^{-1})v({\lambda}) = (1_{\mathcal{M}}-r\lambda_{2}U\mathcal{R}^{-1})v({\lambda}^{\sigma}),
\end{equation}
for all $\lambda \in r\mathbb{D}\times\mathbb{D}$.
Since $\mathcal{R}$ is a diagonal operator on $\mathcal{M}$ and $\lambda_{1}\in r\mathbb{D}$, we obtain
$$\| \lambda_{1}U\mathcal{R}^{-1}\|_{\mathcal{B(M)}}  = \lvert \lambda_{1} \rvert \| \mathcal{R}^{-1} \|_{\mathcal{B(M)}} = \dfrac{\lvert \lambda_{1} \rvert}{r} < 1. $$
Hence $1_{\mathcal{M}}-\lambda_{1}U\mathcal{R}^{-1}$ is invertible. Likewise, since 
$$ \| r \lambda_{2} U\mathcal{R}^{-1}\|_{\mathcal{B(M)}} =r \lvert \lambda_{2}\rvert \| \mathcal{R}^{-1}\|_{\mathcal{B(M)}} = \lvert \lambda_{2} \rvert < 1,$$ 
the operator $1_{\mathcal{M}}-r\lambda_{2}U\mathcal{R}^{-1}$ is also invertible. Note that 
\begin{equation}\label{verifyinvertibility}
(1_{\mathcal{M}}-r\lambda_{2}U\mathcal{R}^{-1})(1_{\mathcal{M}}-\lambda_{1}U\mathcal{R}^{-1}) = (1_{\mathcal{M}}-\lambda_{1}U\mathcal{R}^{-1})(1_{\mathcal{M}}-r\lambda_{2}U\mathcal{R}^{-1}),
\end{equation}
which can be verified by expanding brackets. Multiply both sides of equation (\ref{verifyinvertibility}) on the left and right by $(1_{\mathcal{M}}-\lambda_{1}U\mathcal{R}^{-1})^{-1}$ to produce
\begin{equation}\label{commutingequations}
(1_{\mathcal{M}}-\lambda_{1}U\mathcal{R}^{-1})^{-1}(1_{\mathcal{M}}-r\lambda_{2}U\mathcal{R}^{-1}) = (1_{\mathcal{M}}-r\lambda_{2}U\mathcal{R}^{-1})(1_{\mathcal{M}}-\lambda_{1}U\mathcal{R}^{-1})^{-1}.
\end{equation}
Rearrange equation (\ref{equalitybetweenextension}) to
\begin{equation}\label{firstinversestep}
v({\lambda}) = (1_{\mathcal{M}}-\lambda_{1}U\mathcal{R}^{-1})^{-1}(1_{\mathcal{M}}-r\lambda_{2}U\mathcal{R}^{-1})v({\lambda}^{\sigma}),
\end{equation}
for all $\lambda \in r\mathbb{D}\times\mathbb{D}$.
By equation (\ref{commutingequations}), equation (\ref{firstinversestep}) can be written
\begin{equation}\label{rearrangednewformula-1}
 v({\lambda}) =  (1_{\mathcal{M}}-r\lambda_{2}U\mathcal{R}^{-1})(1_{\mathcal{M}}-\lambda_{1}U\mathcal{R}^{-1})^{-1}v({\lambda}^{\sigma}).
\end{equation}
Thus
\begin{equation}\label{rearrangednewformula}
 (1_{\mathcal{M}}-r\lambda_{2}U\mathcal{R}^{-1})^{-1}v({\lambda}) = (1_{\mathcal{M}}-\lambda_{1}U\mathcal{R}^{-1})^{-1}v({\lambda}^{\sigma}),
\end{equation}
for all $\lambda \in r\mathbb{D}\times\mathbb{D}$.
Let us define $w: r\mathbb{D}\times\mathbb{D}\rightarrow \mathcal{M}$ by
\begin{equation}\label{wequation}
w(\lambda)=(1_{\mathcal{M}}-r\lambda_{2}U\mathcal{R}^{-1})^{-1}v({\lambda})~~\text{for all}~\lambda \in r\mathbb{D}\times \mathbb{D}.
\end{equation}
Note, for $\lambda \in r \mathbb{D}\times\mathbb{D}$,
\begin{align}
v({\lambda}) &= (1_{\mathcal{M}}-r\lambda_{2}U\mathcal{R}^{-1})w(\lambda) \label{vlambda},\\
~v({\lambda}^{\sigma}) &= (1_{\mathcal{M}}-\lambda_{1}U\mathcal{R}^{-1})w(\lambda^{\sigma}) \label{vlambdasigma}.
\end{align}
Thus, by equation (\ref{rearrangednewformula}), for $\lambda \in r \mathbb{D} \times \mathbb{D}$,
\begin{align}
w({\lambda^{\sigma}}) &= (1_{\mathcal{M}}-\lambda_{1}U\mathcal{R}^{-1})^{-1}v({\lambda}^{\sigma}) \nonumber \\
& =  (1_{\mathcal{M}}-r\lambda_{2}U\mathcal{R}^{-1})^{-1}v({\lambda}) \nonumber \\
& = w(\lambda). \label{wsigma}
\end{align}
Hence $w$ is symmetric with respect to the involution $\sigma$ on $r\mathbb{D}\times\mathbb{D}$. Substituting the expressions (\ref{vlambda}) and (\ref{vlambdasigma}) into equation (\ref{averagedequations}) and enlarging $\mathcal{H}$ to $\mathcal{M}$, we find that, for all $\lambda \in r\mathbb{D}\times\mathbb{D}$,
\begin{align*}
1-&\overline{F(\mu)}F(\lambda) \\
= &\Big\langle(1_{\mathcal{M}}-\overline{\mu}_{1}\lambda_{1}\mathcal{R}^{-2})(1_{\mathcal{M}}-r\lambda_{2}U\mathcal{R}^{-1}) w(\lambda),(1_{\mathcal{M}}-r\mu_{2}U\mathcal{R}^{-1})w(\mu)\Big\rangle_{\mathcal{M}}  \\
&+\Big\langle (1_{\mathcal{M}}-r^{2} \overline{\mu}_{2}\lambda_{2}\mathcal{R}^{-2})(1_{\mathcal{M}}-\lambda_{1}U\mathcal{R}^{-1}) w(\lambda), (1_{\mathcal{M}}-\mu_{1}U\mathcal{R}^{-1}) w(\mu) \Big\rangle_{\mathcal{M}}\\
= & \Big\langle(1_{\mathcal{M}}-r\mu_{2}U\mathcal{R}^{-1})^{*}(1_{\mathcal{M}}-\overline{\mu}_{1}\lambda_{1}\mathcal{R}^{-2})(1_{\mathcal{M}}-r\lambda_{2}U\mathcal{R}^{-1}) w(\lambda),w(\mu)\Big\rangle_{\mathcal{M}}  \\
&+\Big\langle (1_{\mathcal{M}}-\mu_{1}U\mathcal{R}^{-1})^{*}(1_{\mathcal{M}}-r^{2} \overline{\mu}_{2}\lambda_{2}\mathcal{R}^{-2})(1_{\mathcal{M}}-\lambda_{1}U\mathcal{R}^{-1}) w(\lambda), w(\mu) \Big\rangle_{\mathcal{M}}\\
= &\Big\langle(1_{\mathcal{M}}-r\overline{\mu_{2}}\mathcal{R}^{-1}U^{*})(1_{\mathcal{M}}-\overline{\mu}_{1}\lambda_{1}\mathcal{R}^{-2})(1_{\mathcal{M}}-r\lambda_{2}U\mathcal{R}^{-1})w(\lambda),w(\mu)\Big\rangle_{\mathcal{M}}  \\
&+\Big\langle (1_{\mathcal{M}}-\overline{\mu_{1}}\mathcal{R}^{-1}U^{*})(1_{\mathcal{M}}-r^{2} \overline{\mu}_{2}\lambda_{2}\mathcal{R}^{-2})(1_{\mathcal{M}}-\lambda_{1}U\mathcal{R}^{-1}) w(\lambda), w(\mu) \Big\rangle_{\mathcal{M}}.
\end{align*}
Thus, for all $\lambda \in r\mathbb{D}\times\mathbb{D}$,
$$ 1-\overline{F(\mu)}F(\lambda) = \langle Z_{r}(\lambda,\mu) w(\lambda),w(\mu)\rangle_{\mathcal{M}}, $$
where
\begin{align*}
Z_{r}(\lambda,\mu) =& (1_{\mathcal{M}}-r\overline{\mu_{2}}\mathcal{R}^{-1}U^{*})(1_{\mathcal{M}}-\overline{\mu}_{1}\lambda_{1}\mathcal{R}^{-2})(1_{\mathcal{M}}- r\lambda_{2}U\mathcal{R}^{-1}) \\
& + (1_{\mathcal{M}}-\overline{\mu_{1}}\mathcal{R}^{-1}U^{*})(1_{\mathcal{M}}-r^{2} \overline{\mu}_{2}\lambda_{2}\mathcal{R}^{-2})(1_{\mathcal{M}}-\lambda_{1}U\mathcal{R}^{-1}).
\end{align*}
Therefore equation (\ref{Zrinnerproduct}) holds.\end{proof}

Observe that the domain $r\cdot\mathbb{G}$ defined in equation \eqref{defrdotG} can be expressed in terms of the symmetrization map $\pi$ by
$$ r \cdot \mathbb{G}:=\pi(r\mathbb{D}\times r\mathbb{D}).$$

\begin{lem}\label{sURanalyticlemma} Let $r\in(0,1)$, let $\mathcal{M}$ be a complex Hilbert space, let   $\mathcal{H}_{1}$ be  a closed non-trivial proper subspace  of $\mathcal{M}$, let 
\begin{equation}\label{RonM}
\mathcal{R}=\begin{bmatrix}
	1_{\mathcal{H}_{1}  } 	& 0                \\
	0 		& r\cdot 1_{\mathcal{H}_{1}^{\perp}}  	\\
\end{bmatrix}\in\mathcal{B(M)},
\end{equation}
let $D$ be a contraction on $\mathcal{M}$ and let $U$ be a unitary operator on $\mathcal{M}$.
\begin{enumerate}
    \item The operator-valued function 
     $$w:r\cdot\mathbb{G}\rightarrow \mathcal{B(M)} : s\mapsto s_{U,\mathcal{R}},$$ 
where, 
for $s=(s_{1},s_{2})\in r\cdot\mathbb{G}$,
\begin{equation}\label{sUR-form}
s_{U,\mathcal{R}} = \bigg(2s_{2}\mathcal{R}^{-1}U-s_{1}\bigg)\bigg(2\mathcal{R}-s_{1}U\bigg)^{-1},
\end{equation}   
is well defined and holomorphic on $r\cdot\mathbb{G}$;
    \item $\|s_{U,\mathcal{R}}\|_{\mathcal{B(M)}}< 1$ for all $s=(s_{1},s_{2})\in r\cdot\mathbb{G}$;
    \item for every $\gamma\in\mathcal{M}$, the $\mathcal{M}$-valued function $$u:r\cdot\mathbb{G}\rightarrow \mathcal{M} \ \text{ defined by} \ u(s) = (1_{\mathcal{M}}-Ds_{U,\mathcal{R}})^{-1}\gamma$$ is holomorphic on $r\cdot\mathbb{G}$.
\end{enumerate}
\end{lem}
\begin{proof}
(1). Let us first check that the definition \eqref{sUR-form} is valid. Since $\mathcal{R}$ is invertible,
$$\bigg(2\mathcal{R}-s_{1}U\bigg) = \bigg(1_{\mathcal{M}}-\dfrac{1}{2}s_{1}U\mathcal{R}^{-1}\bigg)\bigg(2\mathcal{R}\bigg) .$$
 Note that the operator
$$1_{\mathcal{M}}-\dfrac{1}{2}s_{1}U\mathcal{R}^{-1}$$
is invertible in $\mathcal{B(M)}$ for all $s=(s_{1},s_{2})\in r\cdot\mathbb{G}$. Indeed, for $s_{1}=r \lambda_{1}+r\lambda_{2}$ such that $\lambda_{1}\in \mathbb{D}$ and $\lambda_{2}\in\mathbb{D}$, 
$$ \bigg\| \dfrac{1}{2} s_{1}U\mathcal{R}^{-1}\bigg\|_{\mathcal{B(\mathcal{M})}} = \dfrac{1}{2} \lvert s_{1} \rvert \| \mathcal{R}^{-1} \|_{\mathcal{B(\mathcal{M})}} <  \dfrac{1}{2r}\big(2r\big) = 1,$$
therefore the inverse of $1_{\mathcal{M}}-\dfrac{1}{2}s_{1}U\mathcal{R}^{-1}$ exists. Hence,  $\bigg(2\mathcal{R}-s_{1}U\bigg)$
is also invertible in $\mathcal{B(M)}$ for all $s=(s_{1},s_{2})\in r\cdot\mathbb{G}$.
By  \cite[Proposition I.2.6]{completenormedalgebras}, for any $T\in\mathcal{B(M)}$, the map 
$$g:\mathrm{Inv}(\mathcal{B(M)}) \rightarrow \mathrm{Inv}(\mathcal{B(M)}),$$ 
given by $g:T\mapsto T^{-1}$ is holomorphic on $\mathrm{Inv}(\mathcal{B(M))}$.
Therefore, the operator-valued function
$$w:r\cdot\mathbb{G}\rightarrow \mathcal{B(M)} : s \mapsto s_{U,\mathcal{R}},$$ 
where $s_{U,\mathcal{R}} = \bigg(2s_{2}\mathcal{R}^{-1}U-s_{1}\bigg)\bigg(2\mathcal{R}-s_{1}U\bigg)^{-1},$
is holomorphic on $r\cdot \mathbb{G}$. Thus statement (1) is proved.

To prove the second statement, note that
$$
s_{U,\mathcal{R}} = \bigg(s_{2}\mathcal{R}^{-1}U-\dfrac{1}{2}s_{1}\bigg)\bigg(1_{\mathcal{M}}-\dfrac{s_{1}}{2}\mathcal{R}^{-1}U\bigg)^{-1}\mathcal{R}^{-1}.    
$$
Since $s=(s_{1},s_{2})\in r\cdot \mathbb{G}$, there is $q = (q_{1},q_{2}) \in \mathbb{G}$ such that $s_{1}=rq_1$ and $s_2 =r^2 q_2$.
 Thus, for $s=(s_{1},s_{2})\in r\cdot \mathbb{G}$,
\begin{align*}
s_{U,\mathcal{R}} & = \bigg( r^{2}q_{2} 
\begin{bmatrix}
	1_{\mathcal{H}_{1}} & 0 \\
	0 & r^{-1}_{\mathcal{H}_{1}^{\perp}} \\
\end{bmatrix}U-\dfrac{1}{2}q_{1}r\bigg)\bigg(1_{\mathcal{M}}-
\begin{bmatrix}
	1_{\mathcal{H}_{1}} & 0 \\
	0 & r^{-1}_{\mathcal{H}_{1}^{\perp}} \\
\end{bmatrix}\dfrac{1}{2}q_{1}rU\bigg)^{-1}\mathcal{R}^{-1} \\
& =r\bigg( rq_{2}\begin{bmatrix}
	1_{\mathcal{H}_{1}} & 0 \\
	0 & r^{-1}_{\mathcal{H}_{1}^{\perp}} \\
\end{bmatrix}U-\dfrac{1}{2}q_{1}\bigg)\bigg(1_{\mathcal{M}}-
\begin{bmatrix}
	1_{\mathcal{H}_{1}} & 0 \\
	0 & r^{-1}_{\mathcal{H}_{1}^{\perp}} \\
\end{bmatrix}\dfrac{1}{2}q_{1}rU\bigg)^{-1}\mathcal{R}^{-1} \\
& =\bigg( q_{2}\begin{bmatrix}
	r_{\mathcal{H}_{1}} & 0 \\
	0 & 1_{\mathcal{H}_{1}^{\perp}} \\
\end{bmatrix}U-\dfrac{1}{2}q_{1}\bigg)\bigg(1_{\mathcal{M}}-
\begin{bmatrix}
	r_{\mathcal{H}_{1}} & 0 \\
	0 & 1_{\mathcal{H}^{\perp}} \\
\end{bmatrix}\dfrac{1}{2}q_{1}U\bigg)^{-1}\bigg( r\mathcal{R}^{-1}\bigg) \\
& = \bigg( q_{2} 
\begin{bmatrix}
	r_{\mathcal{H}_{1}} & 0 \\
	0 & 1_{\mathcal{H}_{1}^{\perp}} \\
\end{bmatrix}U-\dfrac{1}{2}q_{1}\bigg)\bigg(1_{\mathcal{M}}-
\begin{bmatrix}
	r_{\mathcal{H}_{1}} & 0 \\
	0 & 1_{\mathcal{H}_{1}^{\perp}} \\
\end{bmatrix}\dfrac{1}{2}q_{1}U\bigg)^{-1}\begin{bmatrix}
	r_{\mathcal{H}_{1}} & 0 \\
	0 & 1_{\mathcal{H}_{1}^{\perp}} \\
\end{bmatrix}.
\end{align*}
For all $q=(q_{1},q_{2})\in\mathbb{G}$, define 
$$f_{q}(\lambda) = \dfrac{q_{2}\lambda-\dfrac{1}{2}q_{1}}{1-\dfrac{1}{2}q_{1}\lambda},$$
for $\lambda$ in a neighbourhood of $\overline{\mathbb{D}}$. The linear fractional map $f_{q}$ maps $\mathbb{D}$ onto the open disc with centre and radius
$$ 2\dfrac{\overline{q}_{1}q_{2} -q_{1}}{4-\lvert{q_{1}\rvert}^{2}}, ~~\dfrac{\lvert{q_{1}^{2}-4q_{2}\rvert}}{4-\lvert{q_{1}\rvert}^{2}},$$
respectively. 

Note that  the operator  
$$\begin{bmatrix}
	r\cdot 1_{\mathcal{H}_{1}} & 0 \\
	0 & 1_{\mathcal{H}^{\perp}_{1}} \\
\end{bmatrix}U$$
is a contraction on $\mathcal{M}$ and 
$$ s_{U,\mathcal{R}} = f_{q}\bigg( \begin{bmatrix}
	r\cdot 1_{\mathcal{H}_{1}} & 0 \\
	0 & 1_{\mathcal{H}^{\perp}_{1}} \\
\end{bmatrix}U\Bigg) \begin{bmatrix}
	r\cdot 1_{\mathcal{H}_{1}} & 0 \\
	0 & 1_{\mathcal{H}^{\perp}_{1}} \\
\end{bmatrix}.
$$
By von Neumann's inequality,  we have
\begin{align}
\| s_{U,\mathcal{R}} \|_{\mathcal{B(M)}} &= \Bigg\| f_{q}\bigg( \begin{bmatrix}
	r\cdot 1_{\mathcal{H}_{1}} & 0 \\
	0 & 1_{\mathcal{H}^{\perp}_{1}} \\
\end{bmatrix}U\Bigg) \begin{bmatrix}
	r\cdot 1_{\mathcal{H}_{1}} & 0 \\
	0 & 1_{\mathcal{H}^{\perp}_{1}} \\
\end{bmatrix}\Bigg\| \nonumber\\
& \leq \Bigg\| f_{q}\bigg( \begin{bmatrix}
	r\cdot 1_{\mathcal{H}_{1}} & 0 \\
	0 & 1_{\mathcal{H}^{\perp}_{1}} \\
\end{bmatrix}U\Bigg)  \Bigg\| \nonumber\\
& \leq \sup_{\mathbb{D}} \lvert{f_{q}\rvert} = \dfrac{   2\lvert{\overline{q}_{1}q_{2} -q_{1}\rvert}+\lvert{q_{1}^{2}-4q_{2}\rvert}}{4-\lvert{q_{1}\rvert}^{2}}. \label{supremumsur}
\end{align}
By    \cite[Theorem 2.1]{HGSB}, the right hand side of inequality (\ref{supremumsur}) is less than one for all $ q\in\mathbb{G}$. Thus statement (2) is proved.

(3). In part (1), we have shown that $$w:r\cdot\mathbb{G}\rightarrow \mathcal{B}(\mathcal{M}): \ s \mapsto s_{U,\mathcal{R}}$$ is holomorphic on $r\cdot \mathbb{G}$. Hence,  for every contraction $D \in {\mathcal{B(M)}}$, the map  $s \mapsto 1_{\mathcal{M}}-Ds_{U,\mathcal{R}}$ is holomorphic on $r\cdot\mathbb{G}$. By part (2), for every $s \in r\cdot \mathbb{G}$, $\|s_{U,\mathcal{R}}\|_{\mathcal{B(M)}}< 1$. Thus $1_{\mathcal{M}}-Ds_{U,\mathcal{R}}$ is invertible. 
Therefore, by \cite[Proposition I.2.6]{completenormedalgebras}, for every $\gamma\in\mathcal{M}$, the $\mathcal{M}$-valued function $$u:r\cdot\mathbb{G}\rightarrow \mathcal{M}, \ \text{ defined by} \ u(s) = (1_{\mathcal{M}}-Ds_{U,\mathcal{R}})^{-1}\gamma,$$ is holomorphic on $r\cdot\mathbb{G}$.
\end{proof}

\section{A model formula and a realization for  the symmetrized skew bidisc}

Let us use Theorem \ref{modelformulaforGr} to show that there is a model formula for a function in $\mathcal{S}(\mathbb{G}_{r})$. 
\begin{thm}\label{modelforsur}
Let $r\in(0,1)$ and let $f\in\mathcal{S}(\mathbb{G}_{r})$. Then there exist a model $(\mathcal{M}, (U,\mathcal{R}), u)$ for $f$ on $r\cdot \mathbb{G}$, that is, there exist a complex Hilbert space $\mathcal{M}$,  a closed non-trivial proper subspace  $\mathcal{H}_{1}$ of $\mathcal{M}$,
 a holomorphic map $u:r\cdot \mathbb{G}\rightarrow \mathcal{M}$, a unitary operator $U$ on $\mathcal{M}$ and the operator $\mathcal{R}$ on $\mathcal{M}$ defined by 
\begin{equation}\label{RonM-2}
\mathcal{R}=\begin{bmatrix}
	1_{\mathcal{H}_{1}  } 	& 0                \\
	0 		& r\cdot 1_{\mathcal{H}_{1}^{\perp}}  	\\
\end{bmatrix},
\end{equation} 
such that, for all $s=(s_{1},s_{2})\in r\cdot\mathbb{G}$ and $t=(t_{1},t_{2})\in r \cdot \mathbb{G}$,
\begin{equation}\label{grcorollary}
1- \overline{f(t)}f(s) =\Bigg\langle  \bigg(1_{\mathcal{M}}-t_{U,\mathcal{R}}^{*}s_{U,\mathcal{R}}\bigg)u(s),u(t)\Bigg\rangle_{\mathcal{M}},
  \end{equation}
where the operators $s_{U,\mathcal{R}}$ and $t_{U,\mathcal{R}}$ are strict contractions on $\mathcal{M}$  defined by 
equation \eqref{sUR-form}.
\end{thm}
\begin{remark}
Note that in this theorem we only prove that the formula \eqref{grcorollary} is valid on $ r \cdot \mathbb{G}$, which is a proper subset of $\mathbb{G}_{r}$, since we can only guarantee that  $s_{U,\mathcal{R}}$ given by equation \eqref{sUR-form} and $u$ are well defined  on $ r \cdot \mathbb{G}$.
\end{remark}
\begin{proof}
For the given $f\in\mathcal{S}(\mathbb{G}_{r})$, we define $F=f\circ \pi \circ T_{r} : \mathbb{D}^{2} \rightarrow \overline{\mathbb{D}}$, see equations \eqref{pi-def}, \eqref{Tr-def} and  \eqref{F=fpiTr}. By Theorem \ref{modelformulaforGr}, there exists a Hilbert space $\mathcal{M} = \mathcal{H}_{1}\oplus \mathcal{H}_{2}$, a unitary operator $U$ on $\mathcal{M}$, and  a holomorphic map $w: r\mathbb{D}\times \mathbb{D} \to \mathcal{M}$, which satisfies $w(\lambda^{\sigma}) = w(\lambda)$ for all $\lambda \in r\mathbb{D}\times\mathbb{D}$, such that, for all $\lambda, \mu \in r\mathbb{D}\times\mathbb{D}$,
\begin{equation}\label{rearrangedbidiscformula}
1-\overline{F(\mu)}F(\lambda) = \langle Z_{r}(\lambda, \mu) w(\lambda),w(\mu)\rangle_{\mathcal{M}},
\end{equation}
where
\begin{align}
Z_{r}(\lambda,\mu) = &(1_{\mathcal{M}}-r\overline{\mu_{2}}\mathcal{R}^{-1}U^{*})(1_{\mathcal{M}}-\overline{\mu}_{1}\lambda_{1}\mathcal{R}^{-2})(1_{\mathcal{M}}-r\lambda_{2}U\mathcal{R}^{-1})\nonumber\\
&+ (1_{\mathcal{M}}-\overline{\mu_{1}}\mathcal{R}^{-1}U^{*})(1_{\mathcal{M}}-r^{2} \overline{\mu}_{2}\lambda_{2}\mathcal{R}^{-2})(1_{\mathcal{M}}-\lambda_{1}U\mathcal{R}^{-1}).\label{Zroperator}
\end{align}

Let us rewrite $Z_{r}$ with symmetric variables with respect to $\sigma$ in $r \cdot \mathbb{G}$.
For $\lambda, \mu \in r\mathbb{D}\times\mathbb{D}$, expand equation (\ref{Zroperator}),
\begin{align*}
&Z_{r}(\lambda,\mu) \\
&= (1_{\mathcal{M}}-\overline{\mu_{1}}\lambda_{1}\mathcal{R}^{-2}-r\overline{\mu_{2}}\mathcal{R}^{-1}U^{*} + r\overline{\mu_{1}}\overline{\mu_{2}}\lambda_{1}\mathcal{R}^{-1}U^{*}\mathcal{R}^{-2})(1_{\mathcal{M}}-r\lambda_{2}U\mathcal{R}^{-1})\\
&+(1_{\mathcal{M}}-r^{2}\overline{\mu_{2}}\lambda_{2}\mathcal{R}^{-2}-\overline{\mu_{1}}\mathcal{R}^{-1}U^{*}+r^{2}\overline{\mu_{1}}\overline{\mu_{2}}\lambda_{2}\mathcal{R}^{-1}U^{*}\mathcal{R}^{-2})(1_{\mathcal{M}}-\lambda_{1}U\mathcal{R}^{-1})\\
&= 1_{\mathcal{M}}-r\lambda_{2}U\mathcal{R}^{-1}-\overline{\mu_{1}}\lambda_{1}\mathcal{R}^{-2}+r\overline{\mu_{1}}\lambda_{1}\lambda_{2}\mathcal{R}^{-2}U\mathcal{R}^{-1}-r\overline{\mu_{2}}\mathcal{R}^{-1}U^{*}\\
&+r^{2}\overline{\mu_{2}}\lambda_{2}R^{-1}U^{*}UR^{-1}+r\overline{\mu_{1}}\overline{\mu_{2}}\lambda_{1}\mathcal{R}^{-1}U^{*}\mathcal{R}^{-2}-r^{2}\overline{\mu_{1}}\overline{\mu_{2}}\lambda_{1}\lambda_{2}\mathcal{R}^{-1}U^{*}\mathcal{R}^{-2}U\mathcal{R}^{-1}\\
&+1_{\mathcal{M}}-\lambda_{1}U\mathcal{R}^{-1}-r^{2}\overline{\mu_{2}}\lambda_{2}\mathcal{R}^{-2}+r^{2}\overline{\mu_{2}}\lambda_{1}\lambda_{2}\mathcal{R}^{-2}U\mathcal{R}^{-1}-\overline{\mu_{1}}\mathcal{R}^{-1}U^{*}\\
&+\overline{\mu_{1}}\lambda_{1}\mathcal{R}^{-1}U^{*}U\mathcal{R}^{-1}+r^{2}\overline{\mu_{1}}\overline{\mu_{2}}\lambda_{2}\mathcal{R}^{-1} U^{*}\mathcal{R}^{-2}-r^{2}\overline{\mu_{1}}\overline{\mu_{2}}\lambda_{1}\lambda_{2}\mathcal{R}^{-1}U^{*}\mathcal{R}^{-2} U \mathcal{R}^{-1}.
\end{align*}
Since $U$ is unitary, let us simplify and collect terms to find that 
\begin{align}
Z_{r}(\lambda,\mu) = &2\Big(1_{\mathcal{M}}-r^{2}\overline{\mu}_{1}\overline{\mu}_{2}\lambda_{1}\lambda_{2}\mathcal{R}^{-1}U^{*}\mathcal{R}^{-2}U\mathcal{R}^{-1}\Big)  \nonumber \\
&+\Big( r\lambda_{1}\lambda_{2}(\overline{\mu_{1}}+r\overline{\mu_{2}})\mathcal{R}^{-2}-(\lambda_{1}+r\lambda_{2})\Big)U\mathcal{R}^{-1}  \nonumber\\
&+\mathcal{R}^{-1}U^{*}\Big(r\overline{\mu}_{1} \overline{\mu}_{2}(\lambda_{1}+r\lambda_{2})\mathcal{R}^{-2}-(\overline{\mu_{1}}+r\overline{\mu_{2}})\Big),  \label{Zrnewform}
\end{align}
 for $\lambda,\mu \in r\mathbb{D}\times\mathbb{D}$.

 Thus, for $\lambda, \mu \in r\mathbb{D}\times \mathbb{D}$, we introduce symmetric variables with respect to $\sigma$ 
\begin{align}\label{lambda-s}
& s_{1}=\lambda_{1}+r\lambda_{2},~s_{2} = r\lambda_{1}\lambda_{2} \nonumber\\
& t_{1}=\mu_{1}+r\mu_{2},~t_{2}=r\mu_{1}\mu_{2}. 
\end{align}
It is clear that  $s =(s_1, s_2)$ and $t=(t_1,t_2)$ are in $r \cdot \mathbb{G}$ and, 
\begin{equation}\label{sigmainvol-s1s2}
(s^{\sigma})^{\sigma} = (r s_{2},r^{-1} s_{1})^{\sigma} = (rr^{-1} s_{1},r^{-1} r s_{2}) = s,
\end{equation}
\begin{equation}\label{sigmainvol-t1t2}
(t^{\sigma})^{\sigma} = (r t_{2},r^{-1} t_{1})^{\sigma} = (rr^{-1} t_{1},r^{-1} r t_{2}) = t.
\end{equation}
We can rewrite equation (\ref{Zrnewform}) in terms of $(s_{1},s_{2}), (t_{1},t_{2}) \in r \cdot  \mathbb{G}$ using connections \eqref{lambda-s}, to obtain 
\begin{align}\label{Zr-YRU}
 Z_{r}(\lambda,\mu) = &  Y_{\mathcal{R},U}(s,t) = 2\Big(1_{\mathcal{M}}-\overline{t}_{2}s_{2}\mathcal{R}^{-1}U^{*}\mathcal{R}^{-2}U\mathcal{R}^{-1}\Big)\\
    &+\Big(\overline{t}_{1}s_{2}\mathcal{R}^{-2}-s_{1}\Big)U\mathcal{R}^{-1} 
+\mathcal{R}^{-1}U^{*}\Big(\overline{t}_{2}s_{1}\mathcal{R}^{-2}-\overline{t}_{1}\Big). \nonumber
\end{align}
One can check that
\begin{align*}
    Y_{\mathcal{R},U}(s,t) = &\dfrac{1}{2}\bigg(2-\overline{t}_{1}\mathcal{R}^{-1}U^{*}\bigg)\bigg(2-s_{1}U\mathcal{R}^{-1}\bigg)\\
    &-\dfrac{1}{2}\mathcal{R}^{-1}\bigg(2\overline{t}_{2}U^{*}\mathcal{R}^{-1}-\overline{t}_{1}\bigg)\bigg(2s_{2}\mathcal{R}^{-1}U-s_{1}\bigg)\mathcal{R}^{-1}.
\end{align*}

Recall Definition \ref{sUR-form} of the operator $s_{U,\mathcal{R}}$  on $\mathcal{M}$:
\begin{equation*}
    s_{U,\mathcal{R}} = \bigg(2s_{2}\mathcal{R}^{-1}U-s_{1}\bigg)\bigg(2\mathcal{R}-s_{1}U\bigg)^{-1} 
\end{equation*}
for $s=(s_{1},s_{2}) \in r \cdot \mathbb{G}$.
By Lemma \ref{sURanalyticlemma}, the operator $s_{U,\mathcal{R}}$ is well defined and is a strict contraction for all $s\in r\cdot\mathbb{G}$. We can check that, for $s,t \in r\cdot \mathbb{G}$,
\begin{equation}\label{simplifiedzr}
Y_{\mathcal{R},U}(s,t) = \dfrac{1}{2}\bigg(2-t_{1}U\mathcal{R}^{-1}\bigg)^{*}\bigg(1_{\mathcal{M}}-t_{U,\mathcal{R}}^{*}s_{U,\mathcal{R}}\bigg)\bigg(2-s_{1}U\mathcal{R}^{-1}\bigg).
\end{equation}
Moreover, note that $w$ in equation (\ref{rearrangedbidiscformula}) respects the symmetry of the involution $\sigma$ by equation (\ref{wsigma}). Hence there exists a holomorphic function $x: r\cdot \mathbb{G}\rightarrow \mathcal{M}$ such that, for all $\lambda\in r\mathbb{D}\times\mathbb{D}$, 
$$ w(\lambda) = x(\lambda_{1}+r\lambda_{2}, r\lambda_{1}\lambda_{2}) = x(s_{1},s_{2})=x(s),$$
using the relations  \eqref{lambda-s}.
Recall that for $f\in\mathcal{S}(\mathbb{G}_{r})$, we have defined 
$$F=f\circ \pi \circ T_{r} : \mathbb{D}^{2} \rightarrow \overline{\mathbb{D}},$$ and so, for   $\lambda \in  r\mathbb{D}\times\mathbb{D}$, 
\begin{equation}\label{F-f}
F(\lambda)= f(\lambda_{1}+r\lambda_{2},r\lambda_{1}\lambda_{2})= f(s_1, s_2)= f(s),
\end{equation}
where $s$ is defined by equations \eqref{lambda-s}.
Therefore, using equations \eqref{F-f} and \eqref{Zr-YRU}, we can re-write the equation (\ref{rearrangedbidiscformula})
in the following form
$$ 1-\overline{f(t)}f(s) = \bigg\langle Y_{\mathcal{R},U}(s,t) x(s), x(t)\bigg\rangle_{\mathcal{M}},$$
for all $s,t\in r\cdot\mathbb{G}$.
Hence, by equation (\ref{simplifiedzr}),
$$1-\overline{f(t)}f(s) = \Bigg\langle \dfrac{1}{2} \bigg(2-t_{1}U\mathcal{R}^{-1}\bigg)^{*}\bigg(1_{\mathcal{M}}-t_{U,\mathcal{R}}^{*}s_{U,\mathcal{R}}\bigg)\bigg(2-s_{1}U\mathcal{R}^{-1}\bigg)x(s),x(t)\Bigg\rangle_{\mathcal{M}}$$
and

\begin{align}\label{model-1}
 1 & -  \overline{f(t)}f(s)= \\
 & \Bigg\langle  \bigg(1_{\mathcal{M}}-t_{U,\mathcal{R}}^{*}s_{U,\mathcal{R}}\bigg)\dfrac{1}{\sqrt{2}}\bigg(2-s_{1}U\mathcal{R}^{-1}\bigg)x(s),\dfrac{1}{\sqrt{2}}\bigg(2-t_{1}U\mathcal{R}^{-1}\bigg)x(t)\Bigg\rangle_{\mathcal{M}}, \nonumber
\end{align}
for all $s,t\in r\cdot\mathbb{G}$.
Define a holomorphic map 
$u:r\cdot \mathbb{G}\rightarrow \mathcal{M}$, by 
\begin{equation}
u(s) = \dfrac{1}{\sqrt{2}}\bigg(2-s_{1}U\mathcal{R}^{-1}\bigg)x(s), \ \text{for all} \ s \in r\cdot \mathbb{G}.
\end{equation}
Thus, by equation \eqref{model-1},
$$1-\overline{f(t)}f(s) = \Bigg\langle  \bigg(1_{\mathcal{M}}-t_{U,\mathcal{R}}^{*}s_{U,\mathcal{R}}\bigg)u(s),u(t)\Bigg\rangle_{\mathcal{M}}
\ \text{for all} \ s, t \in r\cdot \mathbb{G}.$$
Therefore equation (\ref{grcorollary}) is proved.\end{proof}

Theorem \ref{modelforsur} allows us to find a realization for functions in $\mathcal{S}(\mathbb{G}_{r})$. 
\begin{thm}\label{realisationformulaforrG}
Let $ r\in (0,1)$ and $f \in \mathcal{S}(\mathbb{G}_{r})$. There exist a scalar $a$,  
 a complex Hilbert space $\mathcal{M}$,  a closed non-trivial proper subspace  $\mathcal{H}_{1}$ of $\mathcal{M}$,
 vectors $\beta, \gamma, \in \mathcal{M}$, operators $D$ and $U$ on $\mathcal{M}$ such that $U$ is unitary and the operator 
\begin{equation}\label{realizationunitaryforgr}
L=\begin{bmatrix}
    a       & 1\otimes \beta \\
    \gamma \otimes 1       & D 
\end{bmatrix}
\end{equation}
is unitary on $\mathbb{C}\oplus \mathcal{M}$ and, for all $s=(s_{1},s_{2})\in r\cdot\mathbb{G}$,
\begin{equation}\label{realization-1}
f(s)=a+\langle s_{U,\mathcal{R}}(1_{\mathcal{M}}-Ds_{U,\mathcal{R}})^{-1}\gamma, \beta \rangle_{\mathcal{M}},
\end{equation}
where  the operator $s_{U,\mathcal{R}}$ is defined by 
equation \eqref{sUR-form} and the operator $\mathcal{R} \in\mathcal{B(M)}$ given by equation \eqref{RonM-2}.
\end{thm}
\begin{proof}
By Theorem \ref{modelforsur}, there exists a Hilbert space $\mathcal{M}=\mathcal{H}_{1}\oplus \mathcal{H}_{2}$, a holomorphic map $u: r\cdot \mathbb{G} \rightarrow \mathcal{M}$, a unitary operator $U$ on $\mathcal{M}$ and an operator $\mathcal{R}\in\mathcal{B}(\mathcal{M})$ given by equation \eqref{RonM-2},
 such that, for all $s,t \in r\cdot\mathbb{G}$, 
\begin{equation}\label{realisationsurfirstequation}
1- \overline{f(t)}f(s) =\Bigg\langle  \bigg(1_{\mathcal{M}}-t_{U,\mathcal{R}}^{*}s_{U,\mathcal{R}}\bigg)u(s),u(t)\Bigg\rangle_{\mathcal{M}}.
\end{equation}
Rearrange equation (\ref{realisationsurfirstequation}) to show that, for all $s,t \in r\cdot\mathbb{G}$, 
$$ 1+ \langle s_{U,\mathcal{R}} u(s),t_{U,\mathcal{R}}u(t)\rangle_{\mathcal{M}} = \langle f(s), f(t) \rangle_{\mathbb{C}}+ \langle u(s), u(t) \rangle_{\mathcal{M}}, $$
which is equivalent to
\begin{equation} 
\Bigg\langle \begin{bmatrix}
	1 \\
	s_{U,\mathcal{R}}u(s)\\
\end{bmatrix}, 
\begin{bmatrix}
	1 \\
	t_{U,\mathcal{R}}u(t)\\
\end{bmatrix}
\Bigg\rangle_{\mathbb{C}\oplus\mathcal{M}} = 
\Bigg\langle\begin{bmatrix}
	f(s) \\
	u(s)\\
\end{bmatrix}, 
\begin{bmatrix}
	f(t) \\
	u(t)\\
\end{bmatrix}
\Bigg\rangle_{\mathbb{C}\oplus\mathcal{M}}.
\end{equation}
This means that the two families of vectors 
\begin{equation*}
	\begin{bmatrix}
	1 \\
	s_{U,\mathcal{R}}u(s) 
	\end{bmatrix}_{s\in r\cdot\mathbb{G}}~\text{and}~
	\begin{bmatrix}
	f(s) \\
	u(s)
	\end{bmatrix}_{s\in r\cdot\mathbb{G}}
\end{equation*}
have the same Gramians in $\mathbb{C}\oplus\mathcal{M}$. Hence there exists a linear isometry $L\in \mathcal{B}(\mathbb{C}\oplus \mathcal{M})$ such that
\begin{equation}\nonumber
L : \overline{\spn}\Bigg\{ \begin{bmatrix}
	1 \\
	s_{U,\mathcal{R}}u(s)\\
\end{bmatrix} : s \in r\cdot \mathbb{G} \Bigg\}\rightarrow 
\overline{\spn}\Bigg\{ \begin{bmatrix}
	f(s) \\
	u(s)\\
\end{bmatrix}: s \in r\cdot \mathbb{G} \Bigg\},
\end{equation}
and
\begin{equation}\label{Lmatrixequation}
L\begin{bmatrix}
	1 \\
	s_{U,\mathcal{R}}u(s)\\
\end{bmatrix} = 
\begin{bmatrix}
	f(s) \\
	u(s)\\
\end{bmatrix},
\end{equation}
for all $s\in r\cdot \mathbb{G}$. Enlarge the Hilbert space $\mathcal{M}$ if necessary, and simultaneously the unitary operator $U$ and the operator $\mathcal{R}$ on $\mathcal{M}$, so that the isometry $L$ extends to a unitary operator
\begin{equation}
	\tilde{L} = \begin{bmatrix}
			a & 1 \otimes \beta \\
			\gamma \otimes 1 & D
			\end{bmatrix},
\end{equation}
on $\mathbb{C}\oplus\mathcal{M}$ for some vectors $\beta, \gamma\in\mathcal{M}$, $a\in \mathbb{C}$ and a contraction $D \in \mathcal{B(M)}$. By equation (\ref{Lmatrixequation}), for every $s\in r\cdot \mathbb{G}$,
\begin{align*}
f(s) &= a+(1\otimes \beta) s_{U,\mathcal{R}}u(s), \\
 u(s)&= (\gamma \otimes 1)(1) + D s_{U,\mathcal{R}}u(s).
\end{align*}
Thus, for every $s\in r\cdot \mathbb{G}$,
\begin{align}
f(s) &= a+\langle s_{U,\mathcal{R}}u(s), \beta \rangle_{\mathcal{M}}, \label{systemofequationsforL}  \\
 u(s)&=\gamma+Ds_{U,\mathcal{R}}u(s). \nonumber  
\nonumber
\end{align}
 Since $D$ is a contraction and by Lemma \ref{sURanalyticlemma}, $\|s_{U,\mathcal{R}}\|_{\mathcal{B(M)}}< 1$ for all $s\in r\cdot \mathbb{G}$, we deduce
 that the operator $(1_{\mathcal{M}}-Ds_{U,\mathcal{R}})$ is invertible for all $s\in r\cdot \mathbb{G}$. Therefore 
 $$ u(s) = (1_{\mathcal{M}}-Ds_{U,\mathcal{R}})^{-1}\gamma, \ \text{for} \ s\in r\cdot \mathbb{G},$$ 
 and so  we can eliminate $u(s)$ from the system of equations (\ref{systemofequationsforL}) to get the following formula 
$$f(s)=a+\langle s_{U,\mathcal{R}}(1_{\mathcal{M}}-Ds_{U,\mathcal{R}})^{-1}\gamma, \beta \rangle_{\mathcal{M}},$$
for all $s\in r\cdot \mathbb{G}$. \end{proof}

We now show that every function $f:r\cdot\mathbb{G} \rightarrow \mathbb{C}$ that has a realization formula \eqref{realization-1} belongs to $\mathcal{S}(r\cdot\mathbb{G})$.
\begin{thm}\label{converseforrdotg}
    Let $\mathcal{M}$ be a complex Hilbert space, let $\mathcal{H}_{1}$ be  a closed non-trivial proper subspace,
       let $\beta, \gamma\in\mathcal{M}$ and let $D$ and $U$ be operators on $\mathcal{M}$ such that $U$ is unitary, the operator 
\begin{equation}\label{thisunitarymatrix}
L=\begin{bmatrix}
    a       & 1\otimes \beta \\
    \gamma \otimes 1       & D 
\end{bmatrix}
\end{equation}
is unitary on $\mathbb{C}\oplus \mathcal{M}$ and let $f:r\cdot \mathbb{G}\rightarrow \mathbb{C}$ be defined by
\begin{equation}\label{realizationformulaforthisproof}
f(s)=a+\langle s_{U,\mathcal{R}}(1_{\mathcal{M}}-Ds_{U,\mathcal{R}})^{-1}\gamma, \beta \rangle_{\mathcal{M}}~~\text{for all}~s\in r\cdot\mathbb{G},
\end{equation}
where 
\begin{equation}\label{surequationrealization}
    s_{U,\mathcal{R}} = \bigg(2s_{2}\mathcal{R}^{-1}U-s_{1}\bigg)\bigg(2\mathcal{R}-s_{1}U\bigg)^{-1}
\end{equation}
and  the operator $\mathcal{R}\in\mathcal{B}(\mathcal{M})$ is given by equation \eqref{RonM-2}.
Then $f\in\mathcal{S}(r\cdot\mathbb{G})$.
\end{thm}
\begin{proof} Let us show that the map $f$ given by equation (\ref{realizationformulaforthisproof}) is well defined and holomorphic on  $r\cdot\mathbb{G}$.  By  Lemma \ref{sURanalyticlemma} (1) and (2),
the operator-valued function 
     $$w:r\cdot\mathbb{G}\rightarrow \mathcal{B(M)} : s\mapsto s_{U,\mathcal{R}},$$  
is well defined and holomorphic on $r\cdot\mathbb{G}$ and $\|s_{U,\mathcal{R}}\|_{\mathcal{B(M)}}< 1$ for all $s=(s_{1},s_{2})\in r\cdot\mathbb{G}$.
Since $L$ is a unitary matrix, $\|D\|_{\mathcal{B}(\mathcal{M})}\leq 1.$ Therefore, by Lemma \ref{sURanalyticlemma} (3),
for every $\gamma\in\mathcal{M}$, the $\mathcal{M}$-valued function $$u:r\cdot\mathbb{G}\rightarrow \mathcal{M} \ \text{ defined by} \ u(s) = (1_{\mathcal{M}}-Ds_{U,\mathcal{R}})^{-1}\gamma$$ is holomorphic on $r\cdot\mathbb{G}$. Hence, $f$ is holomorphic on $r\cdot\mathbb{G}$.

To prove that $\lvert f(s)\rvert \leq 1$ on $r \cdot \mathbb{G}$, note that for all $s\in r\cdot\mathbb{G}$, 
$$ L\begin{bmatrix}
    1 \\
    s_{U,\mathcal{R}}u(s)
\end{bmatrix} = \begin{bmatrix}
    a+(1\otimes\beta)s_{U,\mathcal{R}} \\
    \gamma +Ds_{U,\mathcal{R}}u(s)
\end{bmatrix}= \begin{bmatrix}
    f(s) \\
    u(s)
\end{bmatrix}.$$
Since $L$ is unitary, 
$$\bigg\langle \begin{bmatrix}
    f(s)\\
    u(s)
\end{bmatrix}, \begin{bmatrix}
    f(t)\\
    u(t)
\end{bmatrix} \bigg\rangle_{\mathbb{C}\oplus\mathcal{M}} = \bigg\langle \begin{bmatrix}
    1\\
    s_{U,\mathcal{R}}u(s)
\end{bmatrix}, \begin{bmatrix}
       1\\
    t_{U,\mathcal{R}}u(t)
\end{bmatrix} \bigg\rangle_{\mathbb{C}\oplus\mathcal{M}}~\text{for all}~s,t\in r\cdot\mathbb{G}. $$
By a reshuffle of the above equation, this defines a model $(\mathcal{M},u)$ for the function $f$ on $r\cdot\mathbb{G}$, that is, 
$$1-\overline{f(t)}f(s) = \bigg\langle(1_{\mathcal{M}}-t^{*}_{U,\mathcal{R}}s_{U,\mathcal{R}})u(s),u(t)\bigg\rangle_{\mathcal{M}}~~\text{for}~s,t\in r\cdot\mathbb{G}.$$
Let $t=s$ in the model equation above for $f$. Then
$$1-\lvert f(s) \rvert^{2} =\bigg\langle(1_{\mathcal{M}}-s^{*}_{U,\mathcal{R}}s_{U,\mathcal{R}})u(s),u(s)\bigg\rangle_{\mathcal{M}}.$$
Since $s_{U,\mathcal{R}}$ is a strict contraction for all $s\in r\cdot \mathbb{G}$, we have $1-s^{*}_{U,\mathcal{R}}s_{U,\mathcal{R}}\geq 0$ and thus
$$1-\lvert f(s) \rvert^{2} \geq 0~\text{for all}~~s\in r\cdot \mathbb{G}.$$
Hence $f\in\mathcal{S}(r\cdot\mathbb{G})$.\end{proof}

\section{A realization formula and a Pick theorem for functions in $\mathcal{S}(r\cdot\mathbb{G})$}\label{Pick}

When we restrict our attention to functions from $\mathcal{S}(r\cdot\mathbb{G})$, we can use a different method to derive a realization formula  and a Pick theorem for them.

Let $r\in(0,1)$.  There exists a biholomorphic ``scaling map" between  $\mathbb{G}$ and $r\cdot\mathbb{G}$ 
$$ \psi_r: \mathbb{G} \to r\cdot\mathbb{G} \ \text{given by} \ \psi_r(z_1, z_2)= (rz_1, r^2 z_2).$$
Hence we can deduce a number of statements about  $f\in\mathcal{S}(r \cdot \mathbb{G})$ directly from known facts about holomorphic functions on $\mathbb{G}$, see \cite{AY17}. 

For example, if $f\in\mathcal{S}(r \cdot \mathbb{G})$, then $ f\circ\psi_r \in\mathcal{S}(\mathbb{G})$, and so, by \cite[Theorem 2.2]{AY17},  $ f\circ\psi_r$ has a  $\mathbb{G}$-{model}  $(\m,T,u)$, where $\calm$ is a Hilbert space, $T$ is a contraction acting on $\calm$ and $u:\mathbb{G} \to \m$ is a holomorphic function such that, for all $q,p\in \mathbb{G}$,
\begin{equation}\label{modelform-2}
 1-\overline{f\circ\psi_r (p)}f\circ\psi_r (q)= \ip{ (1-p_T^* q_T) u(q)}{u(p)}_\calm.
\end{equation}
Here, for any point $q=(q_1,q_2)\in \mathbb{G}$ and any contractive linear operator $T$ on a Hilbert space $\calm$, the operator $q_T$ is defined by
\begin{equation}\label{defsU-2}
q_T=(2q_2T-q_1)(2-q_1 T)\inv \quad \mbox{ on } \calm.
\end{equation}
For any $s,t\in r \cdot \mathbb{G}$, apply formula \eqref{modelform-2} to $ q= \psi_r^{-1} (s), p = \psi_r^{-1} (t)$ and observe that
\begin{equation}\label{defsU-3}
q_T= (\psi_r^{-1} (s))_T=(2 r^{-2} s_2 T-r^{-1}s_1)(2-r^{-1}s_1 T)\inv = r^{-1} s_{r^{-1} T}  \quad \mbox{ on } \calm.
\end{equation}
Note that the operator $s_{r^{-1} T}$ is well defined  for $s\in r \cdot \mathbb{G}$ and a contractive linear operator $T$.
Then equation  \eqref{modelform-2} implies that, for all  $s,t\in r \cdot \mathbb{G}$,
\begin{equation}\label{modelform-3}
 1-\overline{f(s)}f(t)= \ip{ (1- r^{-2}  t_{r^{-1} T}^* s_{r^{-1} T}) u(\psi_r^{-1} (s))}{u(\psi_r^{-1} (t))}_\calm.
\end{equation}
Therefore we obtain a model formula $(\m,X,v)$ for $f\in\mathcal{S}(r \cdot \mathbb{G})$, where $\calm$ is a Hilbert space, $X=r^{-1} T $ is an operator acting on $\calm$ with $\|X\| \leq r^{-1}$  and $v:r \cdot \mathbb{G} \to \m$, given by  $v = u \circ \psi_r^{-1}$, is a holomorphic function such that, for all  $s,t\in r \cdot \mathbb{G}$,
\begin{equation}\label{modelform-4}
 1-\overline{f(s)}f(t)= \ip{ (1- r^{-2}  t_{X}^* s_{X}) v(s)}{v (t)}_\calm.
\end{equation}
 
We can also use known facts about functions from $\mathcal{S}(\mathbb{G})$ get a realization formula for functions from 
$\mathcal{S}(r \cdot \mathbb{G})$ and a natural variant of the classical Pick interpolation theorem in which the interpolation nodes lie in $r \cdot \mathbb{G}$.
\begin{thm}\label{interpolationGr} Let  $r\in(0,1)$.
    Let $s_{1},s_{2},\hdots s_{n}$ be distinct points in $r\cdot\mathbb{G}$ and let $w_{1},w_{2},\hdots,w_{n}$ be points in $\overline{\mathbb{D}}$. The following statements are equivalent:  
\begin{enumerate}
    \item There exists a holomorphic function $\varphi\in\mathcal{S}(r\cdot\mathbb{G})$
  such that $\varphi(s_{i})=w_{i}$ for $i=1,\hdots,n$;
    \item There exist a Hilbert space $\mathcal{M}$, an operator $X$  on $\mathcal{M}$ with 
$\|X\| \leq r^{-1}$ and vectors $v_{1},v_{2},\hdots,v_{n}$ in $\mathcal{M}$ such that 
\begin{equation}\label{interpolationmodel}
1-\overline{w_{i}}w_{j}=\bigg\langle \Big(1_{\mathcal{M}}-r^{-2} (s_{i})^{*}_{X} (s_{j})_{X}\Big)v_{j},v_{i}\bigg\rangle_{\mathcal{M}}
\end{equation}
for $i,j=1,\hdots,n$.
\end{enumerate} 
\end{thm}
\begin{proof}
It follows from Theorem 5.1 of \cite{AY17}. \end{proof}

\section{Examples of functions in $\mathcal{S}(r\cdot \mathbb{G})$}

We now make use of the realization formula, Theorem \ref{converseforrdotg}, to give explicit examples of functions in $\mathcal{S}(r\cdot \mathbb{G})$.

\begin{example}\label{rankonematricesgr}{\normalfont{
Let $r\in(0,1)$, let $\mathcal{M}=\mathbb{C}^{2}$ and let $U$ be the unitary operator on $\mathbb{C}^{2}$ given by 
$$U=\begin{bmatrix}
\omega_{1} & 0 \\
0 & \omega_{2} \\
\end{bmatrix}$$
for some  $\omega_{1},\omega_{2}\in\mathbb{T}$. 
Let $a\in\mathbb{C}$, let $\gamma, \beta$ be vectors in $\mathbb{C}^{2}$, 
and let $D=u\otimes v$ be an operator on $\mathbb{C}^{2}$, where  $u, v$ are vectors in $\mathbb{C}^{2}$ with $\|u\|_{\mathbb{C}^{2}}=\|v\|_{\mathbb{C}^{2}}$. Let the operator 
\begin{equation}\label{thisunitarymatrix-2}
L=\begin{bmatrix}
    a       & 1\otimes \beta \\
    \gamma \otimes 1       & D 
\end{bmatrix}
\end{equation}
be unitary on $\mathbb{C}\oplus \mathbb{C}^{2}$. 
Note that since $L$ is unitary, the following conditions on $a, \gamma, \beta, u, v$ are satisfied
\begin{enumerate}
    \item $a=0$;
    \item $\|\gamma\|=\|\beta\|=1$;
    \item $\|u\|=\|v\|=1$;
    \item $\{\gamma,u\}$ and $\{\beta,v\}$ are orthonormal bases of $\mathbb{C}^{2}$.
\end{enumerate}
Then, by Theorem \ref{converseforrdotg}, 
\begin{equation}\label{realiz-example}
f(s)=a+\langle s_{U,\mathcal{R}}(1_{\mathbb{C}^{2}}-Ds_{U,\mathcal{R}})^{-1}\gamma, \beta \rangle_{\mathbb{C}^{2}}, \ \text{for all} \ s\in r\cdot\mathbb{G},
\end{equation}
belongs to $\mathcal{S}(r\cdot \mathbb{G})$. 
Here $s_{U,\mathcal{R}}$ is defined by equation (\ref{surequationrealization}). Let us show that in this case, the function $f$ can be expressed by the following formula 
\small{
\begin{equation}\label{generalformulaforrank1}
    f(s)=\dfrac{\Bigg\langle \begin{bmatrix}
    \phi_{\omega_{1}}(s)(1-u_{2}\overline{v_{2}}r^{-1}\phi_{\omega_{2}r^{-1}}(s)) & r^{-1}u_{1}\overline{v_{2}}\phi_{\omega_{1}}(s)
    \phi_{\omega_{2}r^{-1}}(s)\\
    r^{-1}u_{2}\overline{v_{1}}\phi_{\omega_{1}}(s)\phi_{\omega_{2}r^{-1}}(s) & r^{-1}\phi_{\omega_{2}r^{-1}}(s)
    (1-u_{1}\overline{v_{1}} \phi_{\omega_{1}}(s))\\
\end{bmatrix}\gamma,\beta\Bigg\rangle_{\mathbb{C}^{2}}}{1-u_{1}\overline{v_{1}}\phi_{\omega_{1}}(s)-u_{2}\overline{v_{2}}
r^{-1}\phi_{\omega_{2}r^{-1}}(s)}
\end{equation}  }
for all $s\in r\cdot\mathbb{G}$. Here, for $s=(s_1,s_2)$,
\begin{equation}\label{defszforgr}
\phi_{z}(s) = \dfrac{s_{2}z-\frac{1}{2}s_{1}}{1-\frac{1}{2}s_{1}z}~~\text{for}~z\in\mathbb{C}~\text{such that}~1-\frac{1}{2}s_{1}z\neq0.
\end{equation}
}}
\end{example}
\begin{proof} To use Theorem \ref{converseforrdotg}, we have to be sure that all the parameters given above ensure that the matrix
\begin{equation}
L=\begin{bmatrix}
    a       & 1\otimes \beta \\
    \gamma \otimes 1       & D 
\end{bmatrix}
\end{equation}
is unitary on $\mathbb{C}\oplus \mathbb{C}^{2}$, that is, 
\begin{equation}\label{L-unitary}
LL^{*}=L^{*}L=I_{\mathbb{C}\oplus\mathbb{C}^{2}}.
\end{equation}
We have
\begin{equation}\label{LL*}
LL^{*} = \begin{bmatrix}
    |a|^{2}+\|\beta\|^{2}_{\mathbb{C}^{2}} & a\otimes\gamma+(1\otimes\beta)D^{*}\\
    \overline{a}(\gamma\otimes1)+D(\beta\otimes 1) & \gamma\otimes\gamma +DD^{*}
\end{bmatrix}
\end{equation}
and 
\begin{equation}\label{L*L}
L^{*}L = \begin{bmatrix}
    |a|^{2}+\|\gamma\|^{2}_{\mathbb{C}^{2}} & \overline{a}\otimes\beta+(1\otimes\gamma)D\\
    a(\beta\otimes1)+D^{*}(\gamma\otimes 1) & \beta\otimes\beta +D^{*}D
\end{bmatrix}.
\end{equation}
Since $L$ is unitary, using equations \eqref{LL*} and \eqref{L*L}, we can obtain  the following system of equations 
for $a, \gamma, \beta, u, v$.
\begin{align}
    1 &= \lvert a \rvert^{2}+\|\beta\|^{2} = \lvert a \rvert^{2}+\|\gamma\|^{2} \label{equation1here}\\
    0 &= a\otimes \gamma+\langle v,\beta\rangle_{\mathbb{C}^{2}}(1\otimes u) = \overline{a}\otimes\beta+\langle u,\gamma\rangle_{\mathbb{C}^{2}}(1\otimes v) \label{equation2here}\\
    0 &=\overline{a}\gamma+\langle\beta,v\rangle_{\mathbb{C}^{2}}u = a\beta + \langle \gamma,u \rangle_{\mathbb{C}^{2}} v \label{equationinthislistforunitary}\\
    1_{\mathbb{C}^{2}} &= \gamma\otimes\gamma + \|v\|_{\mathbb{C}^{2}}^{2}(u\otimes u) = \beta\otimes\beta + \|u\|_{\mathbb{C}^{2}}^{2}(v\otimes v).\label{equation3here}
\end{align}
We claim that this system of equations forces: 
\begin{enumerate}
    \item $a=0$;
    \item $\|\gamma\|=\|\beta\|=1$;
    \item $\|u\|=\|v\|=1$;
    \item $\{\gamma,u\}$ and $\{\beta,v\}$ are orthonormal bases of $\mathbb{C}^{2}$.
\end{enumerate}
We prove statement (1) by contradiction. Suppose that $a\neq 0$. From equation (\ref{equationinthislistforunitary}), 
$$0=a\beta+\langle \gamma, u\rangle_{\mathbb{C}^{2}}v.$$
Thus,
$$\beta=-a^{-1}\langle \gamma,u\rangle_{\mathbb{C}^{2}}v$$
and 
$$\beta\otimes\beta = a^{-2}\lvert \langle \gamma,u\rangle_{\mathbb{C}^{2}}\rvert^{2}(v\otimes v).$$
From equation (\ref{equation3here}) with the expression for $\beta\otimes\beta$ above, we have 
$$1_{\mathbb{C}^{2}} = (a^{-2}\lvert \langle \gamma,u\rangle_{\mathbb{C}^{2}}\rvert^{2}+\|u\|_{\mathbb{C}^{2}}^{2})(v\otimes v).$$
This is a contradiction, as $v\otimes v$ is a rank $1$ matrix on $\mathbb{C}^{2}$ and $1_{\mathbb{C}^{2}}$ has rank $2$. Thus $a=0$ necessarily.

Statement (2) follows from equation (\ref{equation1here}), since $a=0$, $\|\gamma\|_{\mathbb{C}^{2}}=\|\beta\|_{\mathbb{C}^{2}}=1$. Moreover, equation (\ref{equation2here}) becomes
$$0 = \langle v,\beta\rangle_{\mathbb{C}^{2}}(1\otimes u) = \langle u,\gamma\rangle_{\mathbb{C}^{2}}(1\otimes v). $$
By the equation above, for all $x\in\mathbb{C}^{2}$,
\begin{align}
    0 &= \langle v,\beta\rangle_{\mathbb{C}^{2}}\langle x,u\rangle_{\mathbb{C}^{2}} \label{orthogonalcondition1} \\
    0 &= \langle u,\gamma\rangle_{\mathbb{C}^{2}}\langle x,v\rangle_{\mathbb{C}^{2}}\label{orthogonalcondition2}.
\end{align}
Equation (\ref{orthogonalcondition2}) implies $u$ is orthogonal to $\gamma$ and equation (\ref{orthogonalcondition1}) implies $v$ is orthogonal to $\beta$. Together, $\{\gamma,u\}$ and $\{\beta,v\}$ are respectively orthogonal in $\mathbb{C}^{2}$. In fact, $\{\gamma,u\}$ and $\{\beta,v\}$ are orthonormal bases of $\mathbb{C}^{2}$; indeed, by equation (\ref{equation3here}), for all $x\in\mathbb{C}^{2}$,
\begin{align}
x &= \langle x,\beta\rangle_{\mathbb{C}^{2}}\beta + \|u\|_{\mathbb{C}^{2}}^{2}\langle x,v \rangle_{\mathbb{C}^{2}} v\label{perpequations1}\\
x &= \langle x,\gamma\rangle_{\mathbb{C}^{2}}\gamma + \|v\|_{\mathbb{C}^{2}}^{2}\langle x,u\rangle_{\mathbb{C}^{2}}u.\label{perpequations2}  
\end{align}
Let $x=v$ in equation (\ref{perpequations1}), we have 
$$v=\|u\|_{\mathbb{C}^{2}}^{2}\|v\|^{2}_{\mathbb{C}^{2}}v.$$
By the assumption $\|u\|_{\mathbb{C}^{2}}=\|v\|_{\mathbb{C}^{2}}$ and by the equation above,
$$1=\|u\|_{\mathbb{C}^{2}}\|v\|_{\mathbb{C}^{2}} = \|u\|_{\mathbb{C}^{2}}^{2}=\|v\|_{\mathbb{C}^{2}}^{2}.$$
Therefore $\|u\|_{\mathbb{C}^{2}}=\|v\|_{\mathbb{C}^{2}}=1$.

We can now utilise the realization formula \eqref{realiz-example}
\begin{equation}\label{realiz-example-2}
f(s)=a+\langle s_{U,\mathcal{R}}(1_{\mathbb{C}^{2}}-Ds_{U,\mathcal{R}})^{-1}\gamma, \beta \rangle_{\mathbb{C}^{2}}, \ \text{for all} \ s\in r\cdot\mathbb{G}.
\end{equation}
Under our assumptions, we have shown that $a$ has to be equal to $0$. By assumption, 
$$D =u\otimes v = \begin{bmatrix}
    u_{1}\overline{v_{1}} & u_{1}\overline{v_{2}} \\
    u_{2}\overline{v_{1}} & u_{2}\overline{v_{2}} \\
\end{bmatrix}.$$
For $U=\begin{bmatrix}
\omega_{1} & 0 \\
0 & \omega_{2} \\
\end{bmatrix}$ and for $s= (s_1, s_2) \in r\cdot\mathbb{G}$,
\begin{align}
s_{U,\mathcal{R}} = & \bigg(2s_{2}\mathcal{R}^{-1}U-s_{1}\bigg)\bigg(2\mathcal{R}-s_{1}U\bigg)^{-1}\\
  =&  \begin{bmatrix}
   \dfrac{s_{2} \omega_1-\frac{1}{2}s_{1}}{1-\frac{1}{2}s_{1} \omega_1}  & 0 \\
    0 &  r^{-1} \dfrac{s_{2}\omega_{2}r^{-1} -\frac{1}{2}s_{1}}{1-\frac{1}{2}s_{1} \omega_{2}r^{-1}   } \\
\end{bmatrix}.
\end{align}
Let us use the notation
$$
\phi_{z}(s) = \dfrac{s_{2}z-\frac{1}{2}s_{1}}{1-\frac{1}{2}s_{1}z} \ \text{for} \ z\in\mathbb{C} \ \text{such that}~1-\frac{1}{2}s_{1}z\neq0.
$$
Thus, for $s= (s_1, s_2) \in r\cdot\mathbb{G}$, 
$$ s_{U,\mathcal{R}} = \begin{bmatrix}
 \phi_{\omega_1}(s)   & 0 \\
    0 &  r^{-1}  \phi_{r^{-1}\omega_2}(s)  \\
\end{bmatrix}.
$$
Therefore
$$1_{\mathbb{C}^{2}}-( u\otimes v )s_{U,\mathcal{R}} = \begin{bmatrix}
    1-u_{1}\overline{v_{1}}\phi_{\omega_{1}}(s) & -u_{1}\overline{v_{2}}r^{-1}\phi_{\omega_{2}r^{-1}}(s)\\
    -u_{2}\overline{v_{1}}\phi_{\omega_{1}}(s) & 1-u_{2}\overline{v_{2}}r^{-1}\phi_{\omega_{2}r^{-1}}(s)
\end{bmatrix}.$$
Note that 
$$\mathrm{det}(1_{\mathbb{C}^{2}}-( u\otimes v )s_{U,\mathcal{R}}) =1-u_{1}\overline{v_{1}}\phi_{\omega_{1}}(s)-u_{2}\overline{v_{2}}r^{-1}\phi_{\omega_{2}r^{-1}}(s).$$
Hence, so long as $\mathrm{det}(1_{\mathbb{C}^{2}}-( u\otimes v )s_{U,\mathcal{R}})\neq 0$, $1_{\mathbb{C}^{2}}-( u\otimes v )s_{U,\mathcal{R}}$ is invertible and is given by 
\begin{align*}
    (1_{\mathbb{C}^{2}} &-( u\otimes v )s_{U,\mathcal{R}})^{-1}  \\   
    &= \big[\mathrm{det}(1_{\mathbb{C}^{2}}-( u\otimes v )s_{U,\mathcal{R}})\big]^{-1}\begin{bmatrix}
    1-u_{2}\overline{v_{2}}r^{-1}\phi_{\omega_{2}r^{-1}}(s) & u_{1}\overline{v_{2}}r^{-1}\phi_{\omega_{2}r^{-1}}(s)\\
    u_{2}\overline{v_{1}}\phi_{\omega_{1}}(s)  & 1-u_{1}\overline{v_{1}}\phi_{\omega_{1}}(s)
\end{bmatrix} \\
&=\dfrac{\begin{bmatrix}
    1-u_{2}\overline{v_{2}}r^{-1}\phi_{\omega_{2}r^{-1}}(s) & u_{1}\overline{v_{2}}r^{-1}\phi_{\omega_{2}r^{-1}}(s) \\
    u_{2}\overline{v_{1}}\phi_{\omega_{1}}(s)  & 1-u_{1}\overline{v_{1}}\phi_{\omega_{1}}(s) 
\end{bmatrix}}{1-u_{1}\overline{v_{1}}\phi_{\omega_{1}}(s) -u_{2}\overline{v_{2}}r^{-1}\phi_{\omega_{2}r^{-1}}(s) }.
\end{align*}
Therefore,  the function $f$ given by equation \eqref{realiz-example-2} is defined by 
\small{
\begin{equation*}
    f(s)=\dfrac{\Bigg\langle \begin{bmatrix}
    \phi_{\omega_{1}}(s) (1-u_{2}\overline{v_{2}}r^{-1}\phi_{\omega_{2}r^{-1}}(s) ) & r^{-1}u_{1}\overline{v_{2}}\phi_{\omega_{1}}(s) \phi_{\omega_{2}r^{-1}}(s) \\
    r^{-1}u_{2}\overline{v_{1}}\phi_{\omega_{1}}(s) \phi_{\omega_{2}r^{-1}}(s)  & r^{-1}\phi_{\omega_{2}r^{-1}}(s) (1-u_{1}\overline{v_{1}}\phi_{\omega_{1}}(s) )\\
\end{bmatrix}\gamma,\beta\Bigg\rangle_{\mathbb{C}^{2}}}{1-u_{1}\overline{v_{1}}\phi_{\omega_{1}}(s) -u_{2}\overline{v_{2}}r^{-1}\phi_{\omega_{2}r^{-1}}(s) }
\end{equation*} }
for all $s\in r\cdot\mathbb{G}$. By  Theorem \ref{converseforrdotg}, this function $f$ belongs to  $\mathcal{S}(r\cdot \mathbb{G})$.\end{proof}

\begin{example}\label{example2} {\normalfont
    For any $r\in(0,1)$ and $\omega\in\mathbb{T}$, the function $\Upsilon_{\omega,r}$ defined by
\begin{equation}\label{grexample}
    \Upsilon_{\omega,r}(s)=\dfrac{s_{2}\omega r^{-1}-\dfrac{1}{2}s_{1}}{1-\dfrac{1}{2}s_{1}\omega r^{-1}}r^{-1}, \ \text{for all} \ 
    s=(s_{1},s_{2})\in r\cdot\mathbb{G},    
\end{equation}
belongs to $\mathcal{S}(r\cdot\mathbb{G})$. 
\begin{proof}
In Example \ref{rankonematricesgr} take $\omega_{1}=\omega_{2}=\omega$ to be complex numbers on the unit circle and the vectors $\beta=\gamma=e_{2}$ and $u=v=e_{1}$, where 
$$e_{1}=\begin{bmatrix}
    1 \\
    0
\end{bmatrix},~e_{2}=\begin{bmatrix}
    0 \\
    1
\end{bmatrix},$$ 
the standard orthonormal bases in $\mathbb{C}^{2}$.
Then $\Upsilon_{\omega,r}\in\mathcal{S}(r\cdot \mathbb{G})$ and has the form given in equation (\ref{grexample}). \end{proof}}
\end{example}

The next example gives us $\Phi_{\omega}$ with $\omega\in\mathbb{T}$, which  is the familiar ``magic function" for $\mathbb{G}$,  see Agler and Young \cite{HGSB}. The functions $\Phi_{\omega}$, $\omega\in\mathbb{T}$, where $\Phi_{\omega}(s,p)= \dfrac{2\omega p-s}{2-\omega s}$ for $(s,p)\in\mathbb{G}$, were called ``magic functions" by Agler in recognition of their power as a tool to prove facts about $\mathbb{G}$. The main application of magic functions in \cite{HGSB,ay2008}, was to identify all automorphisms of $\mathbb{G}$, and they are also central to the solution of the Carath\'{e}odory extremal problem for $\mathbb{G}$.

Note that $\Upsilon_{\omega,r}$ from Example \ref{example2} reduces to the equation 
\begin{equation}\label{Upsilon}
    \Upsilon_{\omega,r}(s)=\Phi_{\omega r^{-1}}(s)r^{-1} \ \text{for all} \    s=(s_{1},s_{2})\in r\cdot\mathbb{G}.    
\end{equation}

\begin{example}{\normalfont{
For any $\omega\in\mathbb{T}$, the function defined by
$$\Phi_{\omega}(s) = \dfrac{s_{2}\omega-\frac{1}{2}s_{1}}{1-\frac{1}{2}s_{1}\omega}$$
for all $s=(s_{1},s_{2})\in\mathbb{G}$, belongs to $\mathcal{S}(\mathbb{G})$, and so to $\mathcal{S}(r\cdot \mathbb{G})$.
}}
\end{example}
\begin{proof}
In Example \ref{rankonematricesgr}, take $\omega_{1}=\omega_{2}=\omega$ to be a complex number on the unit circle, $\gamma=\beta=e_{1}$ and $u=v=e_{2}$, the standard basis of $\mathbb{C}^{2}$. Then the description of the function $f$ from equation (\ref{generalformulaforrank1}) gives us
\begin{align*}
    f(s)&=\dfrac{\Bigg\langle \begin{bmatrix}
    \phi_{\omega}(s) (1-r^{-1}\phi_{\omega r^{-1}}(s) ) & 0\\
    0 & 0\\
\end{bmatrix}e_{1},e_{1}\Bigg\rangle_{\mathbb{C}^{2}}}{(1-r^{-1}\phi_{\omega r^{-1}}(s) )}\\
&=\dfrac{\phi_{\omega}(s)(1-r^{-1}\phi_{\omega r^{-1}}(s) )}{(1-r^{-1}\phi_{\omega r^{-1}}(s) )} \\
&= \phi_{\omega}(s) = \dfrac{s_{2}\omega-\frac{1}{2}s_{1}}{1-\frac{1}{2}s_{1}\omega}\\
&= \Phi_{\omega}(s),
\end{align*}
for all $s=(s_{1},s_{2})\in r\cdot\mathbb{G}$.
It is well known that this function is well defined on $\mathbb{G}$ and belongs to $\mathcal{S}(\mathbb{G})$.\end{proof}

\begin{example}{\normalfont{
For any $\omega_{1},\omega_{2}\in\mathbb{T}$ and $r\in(0,1)$, the function
$$f(s)=\dfrac{r-\sqrt{2}\phi_{\omega_{2}r^{-1}}(s) }{r\sqrt{2}-\phi_{\omega_{2}r^{-1}}(s) }\phi_{\omega_{1}}(s) \ \text{for all} \ 
    s=(s_{1},s_{2})\in r\cdot\mathbb{G},    $$
 belongs to  $\mathcal{S}(r\cdot \mathbb{G})$.
}}
\end{example}
\begin{proof}
Suppose that 
$$\gamma=\dfrac{1}{\sqrt{2}}\begin{bmatrix}
    1 \\
    1
\end{bmatrix}, ~u=\dfrac{1}{\sqrt{2}}\begin{bmatrix}
    -1 \\
    1
\end{bmatrix}$$ 
and $\beta=e_{1}$, $v=e_{2}$ for the standard basis $e_{1},e_{2}$ of $\mathbb{C}^{2}$. By Example \ref{rankonematricesgr}, the function
$$f(s)=\dfrac{r-\sqrt{2}\phi_{\omega_{2}r^{-1}}(s)}{r\sqrt{2}-\phi_{\omega_{2}r^{-1}}(s)}\phi_{\omega_{1}}(s),$$
where $s\in r\cdot\mathbb{G}$ and $\phi_{z}$ is given by formula (\ref{defszforgr}), belongs to $\mathcal{S}(r\cdot\mathbb{G})$.\end{proof}

\section{Declarations}

\noindent {\bf EPSRC grants.}
Evans was supported by the Engineering and Physical Sciences Research Council grant DTP21  EP/T517914/1.
Lykova and Young were partially supported by the Engineering and Physical Sciences Research Council grant EP/N03242X/1.

\noindent {\bf Conﬂict of interest.} The authors have no Conﬂict of interest to declare that are
relevant to the content of this article.

\noindent {\bf Data availability statement.} 
 No data were collected, generated or consulted in connection with this research.\\

\end{document}